\documentclass[11pt,a4paper]{article}
\usepackage[T1]{fontenc}
\usepackage[utf8]{inputenc}
\usepackage{authblk,multirow}
\usepackage{geometry}
\usepackage{color}

\usepackage{amsmath,amsthm}
\usepackage{charter}
\usepackage[charter]{mathdesign}
\usepackage{xcolor}
\usepackage{tikz-cd}
\usepackage{tabularx}
\numberwithin{equation}{section}

\newcommand{\nn}{\nonumber}
\newcommand{\eps}{\varepsilon}
\newcommand{\bx}{{\bf x}}

\newtheorem{theorem}{Theorem}
\newtheorem{lemma}{Lemma}
\newtheorem{remark}{Remark}

\graphicspath{{figures/}}

\author[1]{Yue Feng}
\affil[1]{Laboratoire Jacques-Louis Lions, Sorbonne Universit\'e, Paris 75005, France }

\author[2]{Yichen Guo}
\affil[2]{Department of Mathematics, National University of Singapore, Singapore 119076, Singapore}

\author[3]{Yongjun Yuan}
\affil[3]{MOE-LCSM, School of Mathematics and Statistics, Hunan Normal University, Changsha, Hunan 410081, China}

\date{}

\begin{document}

\title{Uniform error bound of an exponential wave integrator for the long-time dynamics of  the nonlinear Schr\"odinger equation with wave operator}

\maketitle
\begin{abstract}
We establish the uniform error bound of an exponential wave integrator Fourier pseudospectral (EWI-FP) method for the long-time dynamics of the nonlinear Schr\"odinger equation with wave operator (NLSW), in which the strength of the nonlinearity is characterized by $\eps^{2p}$ with $\eps \in (0, 1]$ a dimensionless parameter and $p \in \mathbb{N}^+$. When $0 < \eps \ll 1$, the long-time dynamics of the problem is equivalent to that of the NLSW with $O(1)$-nonlinearity and $O(\eps)$-initial data. The NLSW is numerically solved by the EWI-FP method which combines an exponential wave integrator for temporal discretization with the Fourier pseudospectral method in space. We rigorously establish the uniform $H^1$-error bound of the EWI-FP method at $O(h^{m-1}+\eps^{2p-\beta}\tau^2)$ up to the time at $O(1/\eps^{\beta})$ with $0 \leq \beta \leq 2p$, the mesh size $h$, time step $\tau$ and $m \geq 2$ an integer depending on the regularity of the exact solution. Finally, numerical results are provided to confirm our error estimates of the EWI-FP method and show that the convergence rate is sharp.
\end{abstract}

{\bf Keywords:}  Nonlinear Schr\"odinger equation with wave operator, long-time dynamics, exponential wave integrator, Fourier pseudospectral method, uniform error bound

\maketitle


\section{Introduction}
In this paper, we consider the following nonlinear Schr\"odinger equation with wave operator (NLSW) on the torus $\mathbb{T}^d$ ($d = 1, 2, 3$) 
\begin{equation}
    \label{eq:NLSW_wl_ch4}
    \begin{cases}
        i \partial_t \psi- \alpha \partial_{tt}\psi  + \nabla^2 \psi - \varepsilon^{2p} |\psi|^{2p} \psi= 0, \quad \bx \in \mathbb{T}^d, \quad t > 0, \\
        \psi(\bx,0) = \psi_0(\bx), \quad \partial_t \psi(\bx,0) = \psi_1(\bx), \quad \bx\in \mathbb{T}^d,
    \end{cases}
\end{equation}
where $\psi:=\psi(\bx,t)$ is a complex-valued wave function with the spatial variable $\bx$ and time $t$, $\alpha = O(1)$ is a positive constant and $0<\varepsilon \le 1$ is a dimensionless parameter controlling the strength of the nonlinearity, $p \in \mathbb{N}^+$ and $\nabla^2 = \Delta$ is the $d$-dimensional Laplace operator. In addtiton, $\psi_0(\bx) = O(1)$ and $\psi_1(\bx) = O(1)$ are two given complex-valued functions representing the initial wave and velocity, respectively. The solution of the NLSW with weak nonlinearity \eqref{eq:NLSW_wl_ch4} propagates waves in both space and time with wavelength at $O(1)$ and the wave speed in space is also at $O(1)$. It is well known that the NLSW \eqref{eq:NLSW_wl_ch4} conserves the {\emph{mass}} \cite{BC1,BC2}
\begin{equation*}
N(t) := \int_{\mathbb{T}^d} |\psi(\bx, t)|^2 d\bx - 2\alpha \int_{\mathbb{T}^d} {\rm{Im}}\left(\overline{\psi(\bx, t)}\partial_t \psi(\bx, t) \right) d\bx \equiv N(0), \quad t \ge 0,
\end{equation*}
and the {\emph{energy}}
\begin{align*}
E(t)  := &\ \int_{\mathbb{T}^d}\left[ \alpha|\partial_t\psi(\bx, t)|^2  + |\nabla\psi(\bx, t)|^2 +\frac{\eps^{2p}}{p+1} |\psi(\bx, t)|^{2p+2} \right]d\bx \nn\\
 \equiv & \ E(0), \quad  t \ge 0,
\end{align*}
where $\overline{c}$ and $\rm{Im}(c)$ denote the conjugate and imaginary part of $c$, respectively.

The nonlinear Schr\"odinger equation with wave operator (NLSW) arises from different physical fields including the nonrelativistic limit of the Klein--Gordon equation \cite{MNO,SAY,TSU}, the Langmuir wave envelope approximation in plasma \cite{BCT,CF}, and the modulated planar pulse approximation of the sine-Gordon equation for light bullets \cite{BDX,XIN}. In the past decades, the NLSW \eqref{eq:NLSW_wl_ch4} with $\eps=1$ and $0 < \alpha \ll 1$ has been widely studied analytically and numerically \cite{BC1,BC2,BCT,MNO,SAY}. Along the analytical front, the existence of the solution and the convergence rate to the nonlinear Schr\"odinger equation (NLSE) have been investigated \cite{BCT,MNO,SAY,TSU}. In the numerical aspect, different efficient numerical methods have been proposed and the conservative finite difference methods are most popular \cite{BC1,CHP,DL,GL,WZ,ZC}. In particular, the exponential wave integrator sine pseudospectral (EWI-SP) method has been proposed with optimal uniform error bounds in time established rigorously \cite{BC2}. For more details related to the numerical schemes, we refer to \cite{BZL,GX,HL,JZ,WangJ,WZF,ZHAO} and references therein.

In addition, rescaling the amplitude of the wave function $\psi(\bx,t)$ by introducing a new variable $\phi:=\phi(\bx,t) = \varepsilon \psi(\bx,t)$, the NLSW \eqref{eq:NLSW_wl_ch4} can be reformulated as the following NLSW with $O(1)$-nonlinearity and $O(\eps)$-initial data
\begin{equation}
    \label{eq:NLSW_wi_ch4}
    \begin{cases}
        i \partial_t \phi - \alpha \partial_{tt} \phi + \nabla^2 \phi - |\phi|^{2p} \phi = 0,
        \quad \bx \in \mathbb{T}^d, \quad t>0, \\
        \phi(\bx,0) = \varepsilon \psi_0(\bx), \quad \partial_t \phi(\bx,0) =\varepsilon \psi_1(\bx), \quad \bx\in \mathbb{T}^d.
    \end{cases}
\end{equation}
The long-time dynamics of the NLSW with $O(\eps^{2p})$-nonlinearity and $O(1)$-initial data, i.e., the NLSW \eqref{eq:NLSW_wl_ch4}, is equivalent to that of the NLSW with $O(1)$-nonlinearity and $O(\eps)$-initial data, i.e., the NLSW \eqref{eq:NLSW_wi_ch4}.

In recent years, long-time dynamics of dispersive partial differential equations (PDEs) including the (nonlinear) Schr\"odinger equation, nonlinear Klein--Gordon equation and Dirac equation with weak nonlinearity or small potential are thoroughly studied in the literature \cite{BCF1,BCF2,BFY,FXY,FY,FYJ}. Exponential wave integrators and time-splitting methods are widely used to solve various semilinear evolution equations and  perform well in the long-time simulations \cite{CCO,FAOU,FY,Gau,HLW,HO}. However, to the best of our knowledge, there is no numerical analysis on the error bounds of numerical schemes for the long-time dynamics of the NLSW \eqref{eq:NLSW_wl_ch4} in the literature, especially how the error bound explicitly depends on the mesh size $h$, time step $\tau$ and the small parameter $\eps \in (0, 1]$. Formally, by the energy method and Gronwall inequality, the temporal error bound in the finite time for $t \in [0, T]$ behaves like $O(e^{CT} \tau^2)$ for the second-order temporal discretization, which means that the unbounded temporal error bound is $O(e^{CT/\eps^{\beta}} \tau^2)$ for $t \in [0, T/\eps^{\beta}]$ with $0 \leq \beta \leq 2p$. In order to carry out valid error estimates in the long-time regime, we begin with the proper setup, i.e., the NLSW with weak nonlinearity or small initial data. The aim of this paper is to establish the uniform error bound of the exponential wave integrator Fourier pseudospectral (EWI-FP) method for the long-time dynamics of the NLSW \eqref{eq:NLSW_wl_ch4} up to the time at $O(\eps^{-\beta})$ with $0 \leq \beta \leq 2p$.

The rest of this paper is organized as follows. In section 2, we discuss the derivation of the exponential wave integrator Fourier pseudospectral (EWI-FP) method for the NLSW \eqref{eq:NLSW_wl_ch4} which combines an exponential wave integrator in time with the  Fourier pseudospectral method for spatial discretization. In section 3, we establish the uniform error bound of the EWI-FP method for the long-time dynamics of the NKGE \eqref{eq:NLSW_wl_ch4} up to the time at $O(\eps^{-\beta})$ with $0 \leq  \beta  \leq  2p$. Numerical results are reported in section 4 to confirm our error estimates. Finally, some conclusions are drawn in section 5. Throughout this paper, we adopt the notation $A \lesssim B$ to represent that there exists a generic constant $C > 0$, which is independent of $h$, $\tau$, and $\eps$ such that $|A| \leq  CB$.

\section{An exponential wave integrator Fourier pseudospectral method} 
In this section, we present the exponential wave integrator Fourier pseudospectral (EWI-FP) method to numerically solve the NLSW \eqref{eq:NLSW_wl_ch4}. For simplicity of notations, we only show the numerical scheme for the NLSW \eqref{eq:NLSW_wl_ch4} in one dimension (1D) with $p=1$. It is straightforward to extend it to higher dimensions and/or larger $p$. In 1D, the NLSW \eqref{eq:NLSW_wl_ch4} with $p=1$ on the computational domain $\Omega = (a,b)$ collapses to 
\begin{equation}
    \label{eq:NLSW_wl_1D_ch4}
\left\{
\begin{aligned}
&i \partial_t \psi(x,t) - \alpha \partial_{tt} \psi(x,t) + \partial_{xx} \psi(x,t) - \varepsilon^{2} |\psi(x,t)|^{2} \psi(x,t) = 0, \ x \in \Omega, \ t > 0, \\
 &\psi(x,0) = \psi_0(x), \quad \partial_t \psi(x,0) = \psi_1(x), \quad x\in \overline{\Omega}=[a,b],\\
 &\psi(a,t) = \psi(b,t),\quad \partial_x \psi(a,t) = \partial_x \psi(b,t), \quad t\ge 0.
\end{aligned}\right.
\end{equation}

For an integer $m\ge 0$, we denote by $H^m(\Omega)$ the space of functions $u(x)\in L^2(\Omega)$ with finite $H^m$-norm given by
\begin{equation}
\label{sn}
\|u\|_{H^m}^2=\sum\limits_{l \in \mathbb{Z}} (1+\mu_l^2)^m|\widehat{u}_l|^2,\ \mathrm{for}\ u(x)=\sum\limits_{l\in \mathbb{Z}} \widehat{u}_l e^{i\mu_l(x-a)},\ \mu_l=\frac{2\pi l}{b - a},
\end{equation}
where $\widehat{u}_l \ (l\in \mathbb{Z})$ are the Fourier  coefficients  of the function $u(x)$ \cite{BC2,BFS}. In fact, the  space $H^m(\Omega)$ is the subspace of  classical Sobolev space $W^{m,2}(\Omega)$, which consists of functions with derivatives of order up to $m-1$ being $(b - a)$-periodic \cite{STL}. Since we consider periodic boundary conditions, the above space $H^m(\Omega)$ is suitable. 

Let $\tau = \Delta t > 0$ be the time step size and $t_n = n\tau$ ($n=0,1,\ldots$) as the time steps. Choose the mesh size $h = \Delta x := (b-a)/M$ with $M$ being an even positive integer, then the grid points are denoted as
\begin{equation*}
    x_j :=a+ jh, \quad j \in \mathcal{T}_M^0 =  \{j~|~ j = 0, 1, \ldots, M\}.
\end{equation*}
Denote the index set $\mathcal{T}_M=\left\{l~|~l=-\frac{M}{2},-\frac{M}{2}+1, \ldots, \frac{M}{2}-1\right\}$, and  $C_{\rm per}(\Omega)=\{u\in C(\overline\Omega)~|~u(a) = u(b)\}$ and the spaces
\begin{align*}
& X_M := \left\{u = (u_0, u_1, \ldots, u_M)^T \in \mathbb{C}^{M+1} ~|~u_0 = u_M\right\},\\
& Y_M := \text{span}\left\{e^{i\mu_l (x-a)},\ x\in\overline{\Omega},\ l\in \mathcal{T}_M\right\}.
\end{align*}
For any $u(x) \in C_{\rm per}(\Omega)$ and a vector $u \in X_M$, let $P_M: L^2(\Omega) \to Y_M$ be the standard $L^2$-projection operator onto $Y_M$  and $I_M : C_{\rm per}(\Omega) \to Y_M$ or $I_M : X_M \to Y_M$ be the trigonometric interpolation operator \cite{STL}, i.e.,
\begin{equation*}
P_M u(x) = \sum_{l \in \mathcal{T}_M} \widehat{u}_l e^{i\mu_l(x-a)},\quad I_M u(x) = \sum_{l \in \mathcal{T}_M} \widetilde{u}_l e^{i\mu_l(x-a)},\quad x \in \overline{\Omega},
\end{equation*}
where
\begin{equation*}
\widehat{u}_l = \frac{1}{b - a}\int^{b}_{a} u(x) e^{-i\mu_l (x-a)} dx, \quad \widetilde{u}_l = \frac{1}{M}\sum_{j=0}^{M-1} u_j e^{-i\mu_l (x_j-a)}, \quad l \in \mathcal{T}_M,
\end{equation*}
with $u_j$ interpreted as $u(x_j)$ when involved.

The Fourier spectral discretization for the NLSW \eqref{eq:NLSW_wl_1D_ch4} becomes to find
\begin{equation}
    \label{eq:psi_M_ch4}
    \psi_M(x,t) = \sum_{l \in \mathcal{T}_M} \widehat{\psi}_l(t)  e^{i\mu_l (x-a)}, \quad x\in \Omega, \ t \ge 0,
\end{equation}
such that
\begin{equation}
    \label{eq:NLSW_M_ch4}
   \left( i \partial_t  - \alpha \partial_{tt} \right)\psi_M(x,t) + \partial_{xx} \psi_M(x,t) - \varepsilon^{2} P_M( f(\psi_M(x,t))) = 0, 
\end{equation}
where $f(v)=|v|^2 v$. Plugging \eqref{eq:psi_M_ch4} into \eqref{eq:NLSW_M_ch4}, by the orthogonality of Fourier basis functions, we get
\begin{equation}
    i \frac{d}{dt} \widehat{\psi}_l(t) - \alpha \frac{d^2}{dt^2} \widehat{\psi}_l(t) - |\mu_l|^2 \widehat{\psi}_l(t) - \varepsilon^{2}\widehat{(f(\psi_M))}_l(t) =0, \quad l\in \mathcal{T}_M,\ t > 0.
\end{equation}
For each $l \in \mathcal{T}_M$, when $t$ is near $t_n=n\tau$ ($n \geq 0$), the above ODEs can be rewritten as
\begin{equation}
    \label{eq:NLSW_l_ch4}
    i \frac{d}{ds} \widehat{\psi}_l(t_n+s) - \alpha \frac{d^2}{ds^2} \widehat{\psi}_l(t_n+s) - |\mu_l|^2 \widehat{\psi}_l(t_n+s) - \varepsilon^{2} f_l^n(s) =0, \ s \in \mathbb{R},
\end{equation}
where
\begin{equation}
    f_l^n(s) = \widehat{(f(\psi_M(t_n+s)))}_l.
\end{equation}

Now, we proceed to apply an exponential wave integrator for solving the second-order ODEs \eqref{eq:NLSW_l_ch4}. The variation-of-constant formula or the Duhamel principle shows that for $n \geq 0$,
\begin{align}
    \label{eq:int_form_l_ch4}
        \widehat{\psi}_l(t_n + s)=&\ e^{i \beta^+_l s} \left(-\frac{\beta_l^- \widehat{\psi}_l(t_n) + i \partial_t \widehat{\psi}_l(t_n)}{\beta_l}\right) + e^{i \beta^-_l s}\left(\frac{\beta_l^+ \widehat{\psi}_l(t_n) + i \partial_t \widehat{\psi}_l(t_n)}{\beta_l}\right) \nn \\
        &\ + \frac{i\varepsilon^2}{\alpha \beta_l}\int_0^s \kappa_l(s-w) f_l^n(w) dw, \quad 0 \le s \le \tau,
\end{align}	
where
\begin{equation}
\label{eq:beta_l_ch4}
\begin{split}
&\beta^+_l = \frac{1+\sqrt{1 + 4\alpha |\mu_l|^2}}{2 \alpha}, \\
& \beta^-_l = \frac{1-\sqrt{1 + 4\alpha |\mu_l|^2}}{2 \alpha}= \frac{-2|\mu_l|^2}{1+\sqrt{1 + 4\alpha |\mu_l|^2}} , \\
&\beta_l = \beta^+_l - \beta^-_l = \frac{\sqrt{1 + 4\alpha |\mu_l|^2}}{ \alpha},\\
\end{split}
\end{equation}
and the integral kernel is defined by
\begin{equation}
    \kappa_l(s) = e^{i \beta_l^+ s} - e^{i \beta_l^- s}.
\end{equation}
For $n=0$, taking $s =\tau$ in \eqref{eq:int_form_l_ch4}, the initial condition implies
\begin{align}
    \label{eq:sol_0_ch4}
        \widehat{\psi}_l(\tau) = &\  e^{i \beta^+_l \tau} \left(-\frac{\beta_l^- \widehat{\psi}_l(0) + i \partial_t \widehat{\psi}_l(0)}{\beta_l}\right) + e^{i \beta^-_l \tau}\left(\frac{\beta_l^+ \widehat{\psi}_l(0) + i \partial_t \widehat{\psi}_l(0)}{\beta_l}\right)\nn \\
        &\ + \frac{i\varepsilon^2}{\alpha \beta_l}\int_0^\tau \kappa_l(\tau-w) f_l^0(w) dw\nn \\
        = &\ \frac{\beta_l^+  e^{i \beta^-_l \tau} - \beta_l^-  e^{i \beta^+_l \tau}}{\beta_l}  \widehat{(\psi_0)}_l - i\tau e^{\frac{i\tau}{2\alpha}} \text{sinc} (\tau \beta_l/2)  \widehat{(\psi_1)}_l \nn\\
        &\ + \frac{i\varepsilon^2}{\alpha \beta_l}\int_0^\tau \kappa_l(\tau-w) f_l^0(w) dw, 
\end{align}
with the {\emph{sinc}} function defined as
\begin{equation}
    \text{sinc}(s) = \frac{\sin(s)}{s} \quad \text{for} \quad s\neq 0, \ \text{and}\quad \text{sinc}(0)=1.
\end{equation}
For $n \ge1$, choosing $s = \pm \tau$ in \eqref{eq:int_form_l_ch4} and eliminating the derivative term $\partial_t \widehat{\psi}_l(t_n)$, we have
\begin{align}
    \label{eq:sol_n_ch4}
        \widehat{\psi}_l(t_{n+1})=& -e^{\frac{i\tau}{\alpha}} \widehat{\psi}_l(t_{n-1}) + 2 e^{\frac{i\tau}{2\alpha}} \cos(\beta_l\tau/2) \widehat{\psi}_l(t_{n}) + \frac{i\varepsilon^2}{\alpha \beta_l}\int_0^\tau \kappa_l(\tau-w) f_l^n(w) dw  \nn\\
         & - \frac{i\varepsilon^2e^{\frac{i\tau}{\alpha}}}{\alpha \beta_l}  \int_0^\tau \kappa_l(-w) f_l^{n-1}(w) dw.
\end{align}
For convenience, we introduce the following notations for $l \in \mathcal{T}_M$,
\begin{equation*}
    \sigma_l^+(s) = e^{i\beta_l^+ s/2} \text{sinc}\left(\frac{\beta_l^+ s}{2}\right), \quad \sigma_l^-(s) = e^{i\beta_l^- s/2} \text{sinc}\left(\frac{\beta_l^- s}{2}\right), \quad s\in\mathbb{R}.
\end{equation*}
Then, we are going to approximate the integrals in \eqref{eq:sol_0_ch4} and \eqref{eq:sol_n_ch4}, respectively. For $n=0$, by Taylor expansion, when $l \neq 0$, it leads to
\begin{align}
&\int_0^\tau \kappa_l(\tau-w) f_l^0(w) dw \nn \\
&\approx \int_0^\tau \kappa_l(\tau-w) \left(f_l^0(0) + w \partial_t f_l^0(0)\right) dw \nn \\
&=  \tau f_l^0(0)\left(\sigma_l^+(\tau) - \sigma_l^-(\tau) \right) + i \tau \partial_t f_l^0(0) \left(\frac{1-\sigma_l^+(\tau) }{\beta_l^+}-\frac{1-\sigma_l^-(\tau) }{\beta_l^-}\right),\\
&\int_0^\tau \kappa_l(-w) f_l^0(w) dw \nn   \\
& \approx \tau f_l^0(0) \left(\overline{\sigma_l^+(\tau)} - \overline{\sigma_l^-(\tau)}\right) + i \tau \partial_t f_l^0(0)\left(\frac{1-\sigma_l^+(\tau) }{\beta_l^+ e^{i \beta_l^+ \tau}}-\frac{1-\sigma_l^-(\tau) }{\beta_l^- e^{i \beta_l^- \tau}}\right),
\end{align}
and when $l=0$, as $\beta^-_0 = 0$, we have
\begin{align}
&\int_0^\tau \kappa_0(\tau-w) f_0^0(w) dw  \nn \\
&\approx \int_0^\tau \kappa_0(\tau-w)\left(f_0^0(0) + w \partial_t f_0^0(0)\right) dw \nn \\
&=  \tau f_0^0(0)\left(\sigma_0^+(\tau) - 1\right) + i \tau \partial_t f_0^0(0)\left(\frac{1-\sigma_0^+(\tau) }{\beta_0^+}+\frac{i\tau}{2}\right),\\
&\int_0^\tau \kappa_0(-w) f_0^0(w) dw  \nn  \\
&\approx  \tau f_0^0(0)\left(\overline{\sigma_0^+(\tau)} -1\right) + i \tau \partial_t f_0^0(0) \left(\frac{1-\sigma_0^+(\tau) }{\beta_0^+ e^{i \beta_0^+ \tau}}+\frac{i\tau}{2}\right),
\end{align}
where $\partial_t f_0^0(0)$ can be computed accurately, since $\partial_t \psi(x,0) = \psi_1(x)$ is known in the initial condition.

For $n \ge 1$,  when $l \neq 0$, we apply similar approximations as
\begin{align}
&\int_0^\tau \kappa_l(\tau-w) f_l^n(w) dw  \nn \\
&\approx \int_0^\tau \kappa_l(\tau-w)\left(f_l^n(0) + w \partial_t f_l^n(0)\right) dw \nn \\
&=  \tau f_l^n(0)\left(\sigma_l^+(\tau) - \sigma_l^-(\tau) \right) + i \tau \delta_t^- f_l^n(0) \left(\frac{1-\sigma_l^+(\tau) }{\beta_l^+}-\frac{1-\sigma_l^-(\tau) }{\beta_l^-}\right),\\
&\int_0^\tau \kappa_l(-w) f_l^n(w) dw \nn  \\
&\approx  \tau f_l^n(0)\left(\overline{\sigma_l^+(\tau)} - \overline{\sigma_l^-(\tau)}\right) + i \tau \delta_t^- f_l^n(0) \left(\frac{1-\sigma_l^+(\tau) }{\beta_l^+ e^{i \beta_l^+ \tau}}-\frac{1-\sigma_l^-(\tau)}{\beta_l^- e^{i \beta_l^- \tau}}\right),
\end{align}
and when $l=0$, we have
\begin{align}
&\int_0^\tau \kappa_0(\tau-w) f_0^n(w) dw  \nn \\
&\approx \int_0^\tau \kappa_0(\tau-w) \left(f_0^n(0) + w \partial_t f_0^n(0) \right) dw \nn \\
&=  \tau f_0^n(0) \left(\sigma_0^+(\tau) - 1\right) + i \tau \delta_t^- f_0^n(0) \left(\frac{1-\sigma_0^+(\tau) }{\beta_0^+}+\frac{i\tau}{2}\right),\\
&\int_0^\tau \kappa_0(-w) f_0^n(w) dw  \nn  \\
&\approx \tau f_0^n(0) \left(\overline{\sigma_0^+(\tau)} -1 \right) + i \tau \delta_t^- f_0^n(0) \left(\frac{1-\sigma_0^+(\tau) }{\beta_0^+ e^{i \beta_0^+ \tau}}+\frac{i\tau}{2}\right),
\end{align}
where the finite difference $\delta_t^- f_0^n(0) := (f_0^n(0)-f_0^{n-1}(0))/\tau$ is the approximation of $\partial_t f_0^n(0)$.

Let $\psi_M^n(x)$  be the approximation of $\psi_M(x,t_n)$ and denote $g(\phi(x,t_n))$ for $\phi(x,t)$ as
\begin{equation}
    \label{eq:def_g_ch4}
    g(\phi(0)) = \frac{d}{dt} f(\phi(t)) |_{t=0}, \quad g(\phi(t_n)) = \delta^-_t f(\phi(t_n)), \ n\ge1,
\end{equation}
then the Fourier spectral method can be formulated as follows. Choose $\psi_M^0 = P_M\psi_0$, then we can update the approximation $\psi_M^{n+1} \in Y_M$ for $n \ge 0$ as
\begin{equation}
    \psi_M^{n+1}(x) = \sum_{l \in \mathcal{T}_M} \widehat{(\psi_M^{n+1})}_l e^{i \mu_l (x-a)}, \quad x\in \overline{\Omega},
\end{equation}
with
\begin{align}
\widehat{(\psi_M^{1})}_l = &\ c^0_l \widehat{(\psi^0_M)}_l + d^0_l \widehat{(\psi_1)}_l + p_l \widehat{(f(\psi^0_M))}_l + q_l \widehat{(g(\psi^0_M))}_l,     \label{eq:semi_discretization_scheme_ch4_1}\\
 \widehat{(\psi_M^{n+1})}_l = &\ c_l \widehat{(\psi_M^{n-1})}_l + d_l \widehat{(\psi_M^{n})}_l + p_l \widehat{(f(\psi_M^{n}))}_l + q_l \widehat{(g(\psi_M^{n}))}_l \nn  \\
        &\  - p_l^* \widehat{(f(\psi_M^{n-1}))}_l - q_l^* \widehat{(g(\psi_M^{n-1}))}_l, \quad n\ge 1,    \label{eq:semi_discretization_scheme_ch4_2}
\end{align}
where $g(\psi_M^n)$ ($n \geq 0$) defined in \eqref{eq:def_g_ch4} can be computed by
\begin{equation*}
    \begin{split}
    &g(\psi_M^0) = 2|\psi_M^0|^2 P_M(\psi_1) + (\psi_M^0)^2 \overline{P_M(\psi_1)}, \\&g(\psi_M^n) = \delta_t^- f(\psi_M^n) = \frac{f(\psi_M^n) - f(\psi_M^{n-1})}{\tau}, \quad n\ge 1,
    \end{split}
\end{equation*}
and the coefficients in \eqref{eq:semi_discretization_scheme_ch4_1}--\eqref{eq:semi_discretization_scheme_ch4_2} are given by
\begin{equation}
    \label{eq:coefficients_ch4}
    \begin{split}
       &c^0_l = \frac{\beta_l^+  e^{i \beta^-_l \tau} - \beta_l^-  e^{i \beta^+_l \tau}}{\beta_l}, \quad d^0_l =  -i\tau e^{\frac{i\tau}{2\alpha}} \text{sinc} (\tau \beta_l/2),\\
       &c_l = -e^{\frac{i\tau}{\alpha}}, \quad d_l = 2 e^{\frac{i\tau}{2\alpha}} \cos(\tau \beta_l/2),\\
       &p_l = \frac{i\varepsilon^2 \tau}{\alpha \beta_l} \left(\sigma_l^+(\tau)-\sigma_l^-(\tau)\right) \ (l\neq 0), \quad p_0 = \frac{i\varepsilon^2 \tau}{\alpha \beta_0} \left(\sigma_0^+(\tau)-1\right),\\
       & p_l^* = \frac{i\varepsilon^2 \tau e^{\frac{i\tau}{\alpha}}}{\alpha \beta_l} \left(\overline{\sigma_l^+(\tau)}-\overline{\sigma_l^-(\tau)}\right) \ (l\neq 0), \quad p_0^* = \frac{i\varepsilon^2 \tau e^{\frac{i\tau}{\alpha}}}{\alpha \beta_0} \left(\overline{\sigma_0^+(\tau)}-1\right),\\
       &q_l = \frac{-\varepsilon^2 \tau}{\alpha \beta_l} \left(\frac{1-\sigma_l^+(\tau) }{\beta_l^+ }-\frac{1-\sigma_l^-(\tau) }{\beta_l^-}\right) \ (l\neq 0), \\
       & q_0 = \frac{-\varepsilon^2 \tau}{\alpha \beta_0} \left(\frac{1-\sigma_0^+(\tau) }{\beta_0^+ } +\frac{i\tau}{2}\right), \\
       &q_l^* = \frac{-\varepsilon^2 \tau  e^{\frac{i\tau}{\alpha}}}{\alpha \beta_l} \left(\frac{1-\sigma_l^+(\tau) }{\beta_l^+ e^{i \beta_l^+ \tau}}-\frac{1-\sigma_l^-(\tau) }{\beta_l^-e^{i \beta_l^- \tau}}\right) \ (l\neq 0), \\
       & q_0^* = \frac{-\varepsilon^2 \tau  e^{\frac{i\tau}{\alpha}}}{\alpha \beta_0} \left(\frac{1-\sigma_0^+(\tau) }{\beta_0^+ e^{i \beta_0^+ \tau}} +\frac{i\tau}{2}\right). \\
    \end{split}
\end{equation}
It can be shown by direct computation and \eqref{eq:beta_l_ch4} that $|c_l|, |d_l|, |c_l^0| \lesssim 1$, $|d_l^0| \lesssim \tau^2$, $|p_l|, |p_l^*| \lesssim \varepsilon^2 \tau$ and $|q_l|, |q_l^*| \lesssim \varepsilon^2 \tau^2$ for $l \in \mathcal{T}_M$.

In practical simulations, the above scheme is not suitable due to the difficulty in computing the Fourier coefficients in \eqref{eq:semi_discretization_scheme_ch4_1}--\eqref{eq:semi_discretization_scheme_ch4_2}. As a result, we replace projections by interpolations to get the full discretized EWI-FP scheme. Let $\psi^n_j$ ($n \geq 1$) be the approximations of $\psi(x_j,t_n)$ and choose $\psi_j^0 = \psi_0(x_j)$ for $j \in \mathcal{T}^0_M$, then the numerical approximations $\psi^{n+1} = (\psi^{n+1}_0, \psi^{n+1}_1, \ldots, \psi^{n+1}_M)^T \in X_M$ at $t = t_{n+1}$ (n = 0, 1, \ldots) can be computed by 
\begin{equation}
    \psi_j^{n+1} = \sum_{l \in \mathcal{T}_M} \widetilde{(\psi^{n+1})}_l e^{i \mu_l (x_j-a)}, \quad j \in \mathcal{T}^0_M,
 \label{eq:ewi_fp}   
\end{equation}
with
\begin{align}
\widetilde{(\psi^{1})}_l =&\ c^0_l \widetilde{(\psi_0)}_l + d^0_l \widetilde{(\psi_1)}_l + p_l \widetilde{(f(\psi_0))}_l + q_l \widetilde{(g(\psi_0))}_l,  \quad l \in \mathcal{T}_M,    \label{eq:ewi_fp_scheme_ch4_1} \\
        \widetilde{(\psi^{n+1})}_l =&\ c_l \widetilde{(\psi^{n-1})}_l + d_l \widetilde{(\psi^{n})}_l + p_l \widetilde{(f(\psi^{n}))}_l + q_l \widetilde{(g(\psi^{n}))}_l \nn \\
&\ - p_l^* \widetilde{(f(\psi^{n-1}))}_l - q_l^* \widetilde{(g(\psi^{n-1}))}_l, \quad l \in \mathcal{T}_M, \quad n\ge 1,    \label{eq:ewi_fp_scheme_ch4_2}
\end{align}
where the coefficients $c^0_l$, $d^0_l$,  $c_l$, $d_l$, $p_l$, $p_l^*$, $q_l$ and $q_l^*$ are given in \eqref{eq:coefficients_ch4}, and $g(\psi^n)$ defined by \eqref{eq:def_g_ch4} can be computed by
\begin{equation}
    g(\psi_0) = 2|\psi_0|^2 \psi_1 + (\psi_0)^2 \overline{\psi_1}, \quad g(\psi^n) = \delta_t^- f(\psi^n), \quad n\ge 1.
\end{equation}

The EWI-FP is explicit and can be implemented efficiently thanks to the fast Fourier transform. For each time step, the computational cost is $O(M \log M)$ and the memory cost is $O(M)$.

\section{Uniform error bound for the long-time dynamics}
In this section, we rigorously establish the uniform error bound of the EWI-FP method for the NLSW \eqref{eq:NLSW_wl_1D_ch4} up to the time $t\in [0,T/\varepsilon^{\beta}]$ with $T>0$ fixed and $0 \leq \beta \leq 2$. 

\subsection{Main result}
We assume that for some integer $m \geq 2$, the exact solution $\psi(x,t)$ of the NLSW \eqref{eq:NLSW_wl_1D_ch4} up to the time $T_{\eps} = T/\eps^{\beta}$ satisfies
\begin{equation*}
({\rm A})  \quad   \|\psi\|_{L^{\infty}([0,T_\varepsilon]; W^{1, \infty} \cap H^{m})}+ \|\partial_t \psi\|_{L^{\infty}([0,T_\varepsilon];H^{1})} + \|\partial_{tt} \psi\|_{L^{\infty}([0,T_\varepsilon];H^{1})}  \lesssim 1,
\end{equation*}
and the initial data satisfies
\begin{equation*}
({\rm B}) \quad \|\psi_0(x)\|_{H^{m+2}} \lesssim 1.  
\end{equation*}
Denoting
\begin{align*}
    M_1=\sup_{\eps \in (0,1]}\sup_{t \in [0, T_{\eps}]} \|\psi(\cdot,t)\|_{L^\infty},
\end{align*}
we have the following error estimates for the EWI-FP method.
\begin{theorem}
    \label{thm:main_ch4}
    Let $\psi^n\in X_M$  be the numerical approximation of $ \psi(x,t_n)$ obtained from the EWI-FP \eqref{eq:ewi_fp}--\eqref{eq:ewi_fp_scheme_ch4_2} with fixed $0 < \alpha \le 1$. Under the assumptions {\rm (A)} and {\rm (B)},  there exist constants $ h_0 >0$ and $\tau_0 >0$ sufficiently small and independent of $\varepsilon$, when $0 < h < h_0$, $0 < \tau < \tau_0$, we have
    \begin{equation}
        \label{eq:error_main_ch4}
        \begin{split}
    &\left\|\psi(x, t_n) - I_M\psi^n\right\|_{H^1} \lesssim h^{m-1} + \eps^{2-\beta}\tau^2, \\
    &\left\|\psi^n\right\|_{l^{\infty}}  \le 1 + M_1, \quad 0\le n\le \frac{T/\eps^{\beta}}{\tau},
        \end{split}
    \end{equation}
    for any $0<\varepsilon\le 1$ and $0 \le \beta \le 2$.
\end{theorem}
\begin{remark}
The uniform error bound can be extended to the cases $p \geq 1$ as
\begin{equation}
        \begin{split}
    &\left\|\psi(x, t_n) - I_M\psi^n\right\|_{H^1} \lesssim h^{m-1} + \eps^{2p-\beta}\tau^2, \\
    &\left\|\psi^n\right\|_{l^{\infty}}  \le 1 + M_1, \quad 0\le n\le \frac{T/\eps^{\beta}}{\tau},
        \end{split}
\end{equation}	
for any $0<\varepsilon\le 1$ and $0 \le \beta \le 2p$.
\end{remark}

\begin{remark}
    In 2D/3D case, by the corresponding discrete Sobolev inequalities, Theorem \ref{thm:main_ch4} still holds under the condition $\tau \lesssim \eps^{\beta/2-1}C_d(h)$, where $C_d (h) = 1/|\ln h|$ in 2D and $C_d (h) = h^{1/2}$ in 3D, respectively.
\end{remark}
    
In order to establish the error bound \eqref{eq:error_main_ch4}, we define the error function
 \begin{equation}
 e^n(x) := \psi(x, t_n) - I_M\psi^n, \quad 0 \leq n \leq \frac{T/\eps^{\beta}}{\tau},
 \end{equation}    
then we find from the regularity assumption that \cite{KO,STL}
\begin{align}
\left\|e^n(x)\right\|_{H^1} & \leq \left\|\psi(x, t_n) - P_M\psi(x, t_n)\right\|_{H^1} + \left\|P_M\psi(x, t_n)- I_M\psi^n\right\|_{H^1} \nn\\
&\leq C_1 h^{m-1} + \left\|P_M\psi(x, t_n)- I_M\psi^n\right\|_{H^1},
\end{align}
where $C_1$ is a constant independent of $h, n, \tau$ and $\eps$. Thus, it suffices to study the error $e^n_M \in Y_M$ given as
\begin{equation}
e^n_M(x) = P_M\psi(x,t_n)-I_M\psi^n, \quad 0 \leq n \leq \frac{T/\eps^{\beta}}{\tau}.
\label{eq:def_eM}
\end{equation}

\subsection{Preliminary estimates}
In this subsection, we prepare some results for proving the main theorem. In the following statements, we write $\psi(t)$ for $\psi(x, t)$ in short when there is no confusion.

For $\phi = (\phi_0,\phi_1,\cdots,\phi_M) \in X_M$, define the $l^2$-norm and $l^{\infty}$-norm on $X_M$ as
\begin{equation*}
\left\|\phi\right\|_{l^2}^2 = h\sum_{j=0}^{M-1} |\phi_j|^2, \quad \left\|\phi\right\|_{l^{\infty}} = \max_{0\le j \le M-1} |\phi_j|,
\end{equation*}
then we have the following error estimate. The proof proceeds in the analogous lines as in \cite{BC2,KO,STL} and we omit the details here for brevity.

\begin{lemma}
\label{lmm:h1_ch4}
Let $\phi(x) \in H^1(\Omega)$, $\phi=(\phi_0,\phi_1,\cdots,\phi_M)^T$ with $\phi_j = \phi(x_j)$ ($j \in \mathcal{T}^0_M$), then we have
\begin{equation}
\left\|\delta_x^+ \phi\right\|_{l^2} \lesssim \left\|\nabla I_M \phi\right\|_{L^2} \lesssim \left\|\delta_x^+ \phi\right\|_{l^2},
\end{equation}
where $\delta_x^+ \phi_M = (\phi_1 - \phi_M)/h$.
\end{lemma}

Define the local truncation error  $\xi^{n}(x) \in Y_M$ as
\begin{equation}
    \xi^{n}(x) = \sum_{l \in \mathcal{T}_M} \widetilde{(\xi^n)}_l e^{i\mu_l(x-a)},
\label{eq:le_ch4}
\end{equation}
where
\begin{align}
\label{eq:local_error_ch4_1}
\widetilde{(\xi^0)}_l =&\ \widehat{(\psi(\tau))}_l - c^0_l \widehat{(\psi_0)}_l - d^0_l \widehat{(\psi_1)}_l - p_l \widetilde{(f(\psi_0))}_l - q_l \widetilde{(g(\psi_0))}_l,\\
\widetilde{(\xi^n)}_l =&\ \widehat{(\psi(t_{n+1}))}_l -c_l \widehat{(\psi(t_{n-1}))}_l - d_l \widehat{(\psi(t_n))}_l - p_l \widetilde{(f(\psi(t_n)))}_l \nn\\
&\ - q_l \widetilde{(g(\psi(t_n)))}_l + p_l^* \widetilde{(f(\psi(t_{n-1})))}_l + q_l^* \widetilde{(g(\psi(t_{n-1})))}_l, \quad n\ge 1, \label{eq:local_error_ch4_2}
\end{align}
then we have the following results.
\begin{lemma}
\label{lmm:local_truncation_ch4}
Under the assumptions {\rm (A)} and {\rm (B)}, for $\xi^{n}(x)$  defined in \eqref{eq:le_ch4} with \eqref{eq:local_error_ch4_1}--\eqref{eq:local_error_ch4_2},  we have the decomposition for $l \in \mathcal{T}_M$,
\begin{align}
 \widetilde{(\xi^0)}_l & = e^{i\beta_l^+\tau}\widetilde{(\xi^{0,+})}_l - e^{i\beta_l^-\tau}\widetilde{(\xi^{0,-})}_l ,\\
\widetilde{(\xi^n)}_l &=  e^{i\beta_l^+\tau}\widetilde{(\xi^{n,+})}_l - e^{i\beta_l^-\tau}\widetilde{(\xi^{n,-})}_l - e^{\frac{i\tau}{\alpha}} (\widetilde{(\xi^{n-1,+})}_l - \widetilde{(\xi^{n-1,-})}_l),\quad n\ge 1,  \label{eq:local_error_decomposition_ch4}
\end{align}
and $\xi^{n,\pm}(x) = \sum_{l\in \mathcal{T}_M} \widetilde{(\xi^{n,\pm})}_l e^{i \mu_l(x-a)} \in Y_M$. Then, we have the error bound
    \begin{equation}
        \|\xi^{n,\pm}(x)\|_{H^1} \lesssim \varepsilon^2 \tau(\tau^2 + h^{m-1}).
    \end{equation}
In addition, the coefficients given in \eqref{eq:coefficients_ch4} can be split as
\begin{equation}
\begin{split}
& p_l = e^{i \beta_l^+ \tau} p_l^+ - e^{i \beta_l^- \tau} p_l^-, \quad p_l^* = e^{\frac{i\tau}{\alpha}}\left(p_l^+ -  p_l^-\right) \\
& q_l = e^{i \beta_l^+ \tau} q_l^+ - e^{i \beta_l^- \tau} q_l^-, \quad q_l^* = e^{\frac{i\tau}{\alpha}}\left(q_l^+ -  q_l^-\right), \\
        &p_l^+ = \frac{i\varepsilon^2 \tau}{\alpha \beta_l} \overline{\sigma_l^+(\tau)}, \quad p_l^- = \frac{i\varepsilon^2 \tau}{\alpha \beta_l} \overline{\sigma_l^-(\tau)} \quad (l \neq 0), \quad p_0^- =\frac{i\varepsilon^2 \tau^2}{\alpha \beta_0},\\
        &q_l^+ = \frac{-\varepsilon^2 \tau}{\alpha \beta_l} \cdot \frac{1-\sigma_l^+(\tau) }{\beta_l^+ e^{i \beta_l^+ \tau}}, \\
        &q_l^- = \frac{-\varepsilon^2 \tau}{\alpha \beta_l} \cdot \frac{1-\sigma_l^-(\tau) }{\beta_l^- e^{i \beta_l^- \tau}} \quad (l \neq 0), \quad q_0^- = \frac{-i\varepsilon^2 \tau^2}{2\alpha \beta_0},	
\end{split}
\label{eq:decomposition_coefficients_ch4}
\end{equation}
where $|p_l^{\pm}| \lesssim \varepsilon^2 \tau$, $|q_l^{\pm}| \lesssim \varepsilon^2 \tau^2$. 
\end{lemma}
\begin{proof}
In \eqref{eq:semi_discretization_scheme_ch4_1}--\eqref{eq:semi_discretization_scheme_ch4_2}, we replace $\psi_M(x, t)$ by $\psi(x, t)$, the equations still hold for $l \in \mathcal{T}_M$. We use the same notation $f_l^n(s) = \widetilde{(f(\psi(t_n+s)))}_l$ without confusion. Substituting \eqref{eq:semi_discretization_scheme_ch4_1}--\eqref{eq:semi_discretization_scheme_ch4_2} into \eqref{eq:local_error_ch4_1}--\eqref{eq:local_error_ch4_2}, we know for $l\in \mathcal{T}_M$,
\begin{align}
\label{eq:local_error_2_ch4_1}
\widetilde{(\xi^0)}_l =&\ \frac{i\varepsilon^2}{\alpha \beta_l}\int_0^\tau \kappa_l(\tau-w) f_l^0(w) dw - p_l \widetilde{(f(\psi_0))}_l - q_l \widetilde{(g(\psi_0))}_l,\\
            \widetilde{(\xi^n)}_l =&\  \frac{i\varepsilon^2}{\alpha \beta_l}\int_0^\tau \kappa_l(\tau-w) f_l^n(w) dw - \frac{i\varepsilon^2}{\alpha \beta_l} e^{\frac{i\tau}{\alpha}} \int_0^\tau \kappa_l(-w) f_l^{n-1}(w) dw \nn \\ 
            &\ - p_l \widetilde{(f(\psi(t_n)))}_l- q_l \widetilde{(g(\psi(t_n)))}_l \nn\\
            &\ + p_l^* \widetilde{(f(\psi(t_{n-1})))}_l + q_l^* \widetilde{(g(\psi(t_{n-1})))}_l, \quad n\ge 1. \label{eq:local_error_2_ch4_2}
\end{align}
Denote the integral approximation errors $Q^{n,\pm}(x) = \sum_{l\in \mathcal{T}_M} \widetilde{(Q^{n,\pm})}_l e^{i \mu_l (x-a)} \in Y_M$ as  
\begin{equation}
\widetilde{(Q^{n,\pm})}_l = \frac{i\varepsilon^2}{\alpha \beta_l}\int_0^\tau e^{-i \beta_l^\pm w} f_l^n(w)dw  \mp (p_l^\pm \widehat{(f(\psi(t_n)))}_l +  q_l^\pm \widehat{(g(\psi(t_n)))}_l), 
 \label{eq:def_Q}      
\end{equation}
and the interpolation errors $P^{n,\pm}(x) = \sum_{l\in \mathcal{T}_M} \widetilde{(P^{n,\pm})}_l e^{i \mu_l (x-a)} \in Y_M$ as
\begin{align}
\label{eq:def_P1}
\widetilde{(P^{0,\pm})}_l =&\ p_l^\pm (\widehat{(f(\psi_0))}_l-\widetilde{(f(\psi_0))}_l)+q_l^\pm(\widehat{(2|\psi_0|^2 \psi_1)}_l \nn \\
            &\ + \widehat{((\psi_0)^2 \overline{\psi_1})}_l-\widetilde{(2|\psi_0|^2 \psi_1)}_l - \widetilde{((\psi_0)^2 \overline{\psi_1})}_l),  \\
\widetilde{(P^{n,\pm})}_l =&\ p_l^\pm (\widehat{(f(\psi(t_n)))}_l-\widetilde{(f(\psi(t_n)))}_l) \nn\\
& \ +q_l^\pm( \widehat{(\delta_t^- f(\psi(t_n)))}_l-\widetilde{(\delta_t^- f(\psi(t_n)))}_l). \label{eq:def_P2}
\end{align}
Combining \eqref{eq:decomposition_coefficients_ch4} with \eqref{eq:local_error_2_ch4_1}--\eqref{eq:local_error_2_ch4_2}, and defining
    \begin{equation}
        \xi^{n,\pm}(x) = Q^{n,\pm}(x) + P^{n,\pm}(x),
    \end{equation}
the decomposition \eqref{eq:local_error_decomposition_ch4} holds. Now, in order to estimate $\xi^{n,\pm}(x)$, we only need to estimate $Q^{n,\pm}(x)$ and $P^{n,\pm}(x)$, respectively.

We begin with the estimates of $Q^{n,\pm}(x)$. By the definition \eqref{eq:def_Q}, we have
\begin{align*}
\widetilde{(Q^{0,\pm})}_l = & \ \frac{i\varepsilon^2}{\alpha \beta_l}\int_0^\tau e^{-i \beta_l^\pm w} (f_l^0(w) - f_l^0(0) - w\partial_t f_l^0(0))dw \nn \\
            = & \ \frac{i\varepsilon^2}{\alpha \beta_l}\int_0^\tau \int_0^{w} \int_0^{w_1} e^{-i \beta_l^\pm w} (\partial_{tt} f_l^0(w_2)) dw_2 dw_1 dw, \\ 
\widetilde{(Q^{n,\pm})}_l =&\ \frac{i\varepsilon^2}{\alpha \beta_l}\int_0^\tau e^{-i \beta_l^\pm w} (f_l^n(w) - f_l^n(0) - w\delta_t^- f_l^n(0))dw \nn \\
            = &\ \frac{i\varepsilon^2}{\alpha \beta_l}\int_0^\tau \int_0^{w} \int_0^{w_1} e^{-i \beta_l^\pm w} (\partial_{tt} f_l^n(w_2)) dw_2 dw_1 dw \nn \\
            &\ -\frac{i\varepsilon^2}{\alpha \beta_l}\int_0^\tau e^{-i \beta_l^\pm w} w \left(\int_0^{\tau} \int_0^{w_1} \partial_{tt} f_l^n(-w_2) dw_2 dw_1\right) dw, \quad n \geq 1,
\end{align*}
which imply 
\begin{align}
|\widetilde{(Q^{0,\pm})}_l| \lesssim &\  \varepsilon^2 \int_0^\tau \int_0^{w} \int_0^{w_1} |\widehat{(\partial_{tt} f(\psi(w_2)))}_l| dw_2 dw_1 dw,\\
|\widetilde{(Q^{n,\pm})}_l| \lesssim &\ \varepsilon^2 \int_0^\tau \int_0^{w} \int_0^{w_1} |\widehat{(\partial_{tt} f(\psi(t_n+w_2)))}_l| dw_2 dw_1 dw \nn \\
            &\ +\varepsilon^2 \int_0^\tau \int_0^{\tau} \int_0^{w_1} |\widehat{(\partial_{tt} f(\psi(t_n-w_2)))}_l| dw_2 dw_1 dw, \quad n\geq 1.
\end{align}
Under the assumptions (A) and (B), $\|\psi(x,t)\|_{H^1}, \|\partial_t \psi(x,t)\|_{H^1}, \|\partial_{tt} \psi(x,t)\|_{H^1} \lesssim 1$ for $t\in [0,T_{\varepsilon}]$. Noticing $f(\phi) = |\phi|^2 \phi$ is smooth, direct computation shows that
    \begin{equation}
        \|\partial_{tt} f(\psi(t))\|_{H^1} \lesssim 1, \quad  t\in [0,T_{\varepsilon}].
    \end{equation}
By Cauchy inequality and Bessel inequality, we have for $n=0$,
\begin{align*}
\|Q^{0,\pm}(x)\|_{H^1}^2 &= (b-a) \sum_{l\in \mathcal{T}_M} (1+ |\mu_l|^2) |\widetilde{(Q^{0,\pm})}_l|^2 \\
&\lesssim \varepsilon^4 \sum_{l\in \mathcal{T}_M} (1+ |\mu_l|^2) (\int_0^\tau \int_0^{w} \int_0^{w_1} |\widehat{(\partial_{tt} f(\psi(w_2)))}_l| dw_2 dw_1 dw)^2\\
&\lesssim \varepsilon^4  \tau^3 \int_0^\tau \int_0^{w} \int_0^{w_1} \sum_{l\in \mathcal{T}_M} (1+ |\mu_l|^2) |\widehat{(\partial_{tt} f(\psi(w_2)))}_l|^2 dw_2 dw_1 dw\\
&\lesssim \varepsilon^4  \tau^3 \int_0^\tau \int_0^{w} \int_0^{w_1} \|\partial_{tt} f(\psi(w_2))\|_{H^1}^2 dw_2 dw_1 dw\\
&\lesssim \varepsilon^4  \tau^6.
\end{align*}
Similarly, for $n \ge 1$, it leads to
\begin{align*}
\|Q^{n,\pm}(x)\|_{H^1}^2 =&\ (b-a) \sum_{l\in \mathcal{T}_M} (1+ |\mu_l|^2) |\widetilde{(Q^{n,\pm})}_l|^2 \\
\lesssim &\ \varepsilon^4 \sum_{l\in \mathcal{T}_M} (1+ |\mu_l|^2) \biggl(\int_0^\tau \int_0^{w} \int_0^{w_1} \left|\widehat{(\partial_{tt} f(\psi(t_n + w_2)))}_l\right| dw_2 dw_1 dw\\
&\ +\int_0^\tau \int_0^{\tau} \int_0^{w_1} \left|\widehat{(\partial_{tt} f(\psi(t_n-w_2)))}_l\right| dw_2 dw_1 dw \biggr)^2\\
\lesssim &\  \eps^4  \tau^3 \biggl( \int_0^\tau \int_0^{w} \int_0^{w_1} \left\|\partial_{tt} f(\psi(t_n+w_2))\right\|_{H^1}^2 dw_2 dw_1 dw\\
&\ + \int_0^\tau \int_0^{\tau} \int_0^{w_1} \left\|\partial_{tt} f(\psi(t_n-w_2))\right\|_{H^1}^2 dw_2 dw_1 dw \biggr)\\
\lesssim & \  \eps^4  \tau^6.
\end{align*}
Thus, we obtain
\begin{equation}
\left\|Q^{n,\pm}(x)\right\|_{H^1} \lesssim \varepsilon^2  \tau^3, \quad 0 \leq n \leq \frac{T/\eps^{\beta}}{\tau}.
\label{eq:err_Q}
\end{equation}

Next,  we are going to estimate $P^{n,\pm}(x)$. By the definition \eqref{eq:def_P1}--\eqref{eq:def_P2}, we have
\begin{align*}
\left\|P^{0,\pm}(x)\right\|_{H^1} \le & \ p_l^\pm \left\|P_M (f(\psi_0)) - I_M (f(\psi_0))\right\|_{H^1} \\
&\ + q_l^\pm \left(\left\|P_M(2|\psi_0|^2 \psi_1) - I_M(2|\psi_0|^2 \psi_1)\right\|_{H^1}\right. \\
&\ + \left.\left\|P_M((\psi_0)^2 \overline{\psi_1})) - I_M((\psi_0)^2 \overline{\psi_1})\right\|_{H^1}\right),
\end{align*}
and for $n\ge 1$,
\begin{align*}
\left\|P^{n,\pm}(x)\right\|_{H^1} \le &\  p_l^\pm \left\|P_M (f(\psi(t_n)))- I_M (f(\psi(t_n)))\right\|_{H^1} \\
&\ + q_l^\pm \left\|P_M (\delta_t^- f(\psi(t_n)))- I_M (\delta_t^- f(\psi(t_n)))\right\|_{H^1}.\\
\end{align*}
Since $\psi(x,t), \psi_0(x), \psi_1(x)\in H^m$, it is obvious that $|\psi_0|^2 \psi_1$, $(\psi_0)^2 \overline{\psi_1}$, $f(\psi(t_n))$,  $\delta_t^-f(\psi(t_n))$ $\in  H^m$. Noticing $p_l^\pm \lesssim \varepsilon^2 \tau$, $q_l^\pm \lesssim \varepsilon^2 \tau^2$, we have
\begin{equation}
\left\|P^{n,\pm}(x)\right\|_{H^1} \lesssim \varepsilon^2 (\tau h^{m-1} + \tau^2  h^{m-1})\lesssim \varepsilon^2\tau h^{m-1}, \quad 0 \leq n \leq \frac{T/\eps^{\beta}}{\tau}.
\label{eq:err_P}
\end{equation}
Combining the error bound \eqref{eq:err_Q} for $Q^{n,\pm}(x)$ and the error bound \eqref{eq:err_P} for $P^{n,\pm}(x)$, we obtain
    \begin{equation}
        \|\xi^{n,\pm}(x)\|_{H^1} \lesssim \varepsilon^2 \tau(\tau^2 + h^{m-1}), \quad 0 \leq n \leq \frac{T/\eps^{\beta}}{\tau}.
    \end{equation}
\end{proof}

By \eqref{eq:def_eM} and \eqref{eq:le_ch4}, we have 
\begin{align}
\label{eq:em_decomposition_ch4}
\widetilde{(e_M^{1})}_l =&\ \widetilde{(\xi^0)}_l + c^0_l( \widehat{(\psi_0)}_l -\widetilde{(\psi_0)}_l) + d^0_l (\widehat{(\psi_1)}_l-\widetilde{(\psi_1)}_l),\\
\widetilde{(e_M^{n+1})}_l =&\ \widetilde{(\xi^n)}_l  + c_l (\widehat{(\psi(t_{n-1}))}_l - \widetilde{(\psi^{n-1})}_l) + d_l (\widehat{(\psi(t_n))}_l - \widetilde{(\psi^n)}_l) \nn \\
&\  + p_l \left(\widetilde{(f(\psi(t_n)))}_l - \widetilde{(f(\psi^n))}_l\right) + q_l \left(\widetilde{(g(\psi(t_n)))}_l - \widetilde{(g(\psi^n))}_l\right) \nn\\
&\ - p_l^*\left(\widetilde{(f(\psi(t_{n-1})))}_l - \widetilde{(f(\psi^{n-1}))}_l\right) - q_l^*  \left(\widetilde{(g(\psi(t_{n-1})))}_l - \widetilde{(g(\psi^{n-1}))}_l\right) \nn\\
=&\ \widetilde{(\xi^n)}_l + \widetilde{(\eta^n)}_l + c_l \widetilde{(e_M^{n})}_l + d_l \widetilde{(e_M^{n-1})}_l,  \quad n\ge 1,
\end{align}
where $\eta^n(x) = \sum_{l\in \mathcal{T}_M} \widetilde{(\eta^n)}_l e^{i \mu_l (x-a)} \in Y_M$ for $n \geq 1$ is defined by
\begin{align}
\label{eq:eta_ch4}
\widetilde{(\eta^n)}_l =&\ p_l (\widetilde{(f(\psi(t_n)))}_l - \widetilde{(f(\psi^n))}_l) + q_l (\widetilde{(g(\psi(t_n)))}_l - \widetilde{(g(\psi^n))}_l) \nn \\
        &\ - p_l^* (\widetilde{(f(\psi(t_{n-1})))}_l - \widetilde{(f(\psi^{n-1}))}_l) - q_l^*  (\widetilde{(g(\psi(t_{n-1})))}_l - \widetilde{(g(\psi^{n-1}))}_l),
\end{align}
then we have the following decomposition and estimates for $\eta^{n}(x)$.

\begin{lemma}
\label{lmm:stability_ch4}
Under assumptions (A) and (B), for $l \in \mathcal{T}_M$, we have the following decomposition
    \begin{equation}
        \label{eq:eta_decomposition_ch4}
            \widetilde{(\eta^{n})}_l =e^{i\beta_l^+\tau}\widetilde{(\eta^{n,+})}_l - e^{i\beta_l^-\tau}\widetilde{(\eta^{n,-})}_l - e^{\frac{i\tau}{\alpha}} (\widetilde{(\eta^{n-1,+})}_l - \widetilde{(\eta^{n-1,-})}_l), \quad n\ge 1,
    \end{equation}
and $\eta^{n,\pm}(x) = \sum_{l\in \mathcal{T}_M} \widetilde{(\eta^{n,\pm})}_l e^{i \mu_l(x-a)} \in Y_M$. Assume $\|\psi^n\|_{l^{\infty}} \le M_1 +1$, we have the estimate
\begin{equation}
\label{eq:eta_lmm_ch4}
\left\|\eta^{n,\pm}(x)\right\|_{H^1} \lesssim \varepsilon^2 \tau\left(\left\|e^{n-1}_M\right\|_{H^1} + \left\|e^n_M\right\|_{H^1} + h^{m-1}\right), \quad n \ge 1.
    \end{equation}
\end{lemma}

\begin{proof}
\normalfont
Define $\eta^{n,\pm}(x)$ as
\begin{equation}
    \begin{split}
        \widetilde{(\eta^{n,\pm})}_l &= p_l^{\pm} (\widetilde{(f(\psi(t_n)))}_l - \widetilde{(f(\psi^n))}_l) + q_l^{\pm} (\widetilde{(g(\psi(t_n)))}_l - \widetilde{(g(\psi^n))}_l),
    \end{split}
\end{equation}
then it is easy to verify the decomposition \eqref{eq:eta_decomposition_ch4} holds. Noticing $\widetilde{(f(\psi(t_0)))}_l = \widetilde{(f(\psi^0))}_l$, $\widetilde{(g(\psi(t_0)))}_l = \widetilde{(g(\psi^0))}_l$, and the definition of $g$ for $n\ge 1$ \eqref{eq:def_g_ch4}, we have
\begin{equation}
    \begin{split}
        \widetilde{(\eta^{n,\pm})}_l = \begin{cases}
             &0, \quad n=0,\\
             &p_l^{\pm} (\widetilde{(f(\psi(t_n)))}_l - \widetilde{(f(\psi^n))}_l) \\
              &\quad + q_l^{\pm} (\widetilde{(\delta_t^- f(\psi(t_n)))}_l - \widetilde{(\delta_t^- f(\psi^n))}_l), \quad n \ge 1.
        \end{cases}
    \end{split}
\end{equation}
By Parseval equality and Cauchy inequality, noticing $|p_l^{\pm}| \lesssim \varepsilon^2 \tau$, $|q_l^{\pm}| \lesssim \varepsilon^2 \tau^2$, we have
\begin{align}
\label{eq:eta_estimate_1_ch4}
\left\|\eta^{n,\pm}(x)\right\|_{H_1}^2 =&\ (b-a) \sum_{l \in \mathcal{T}_M} (1 + |\mu_l|^2) |\widetilde{(\eta^{n,\pm})}_l|^2 \nn\\
\lesssim &\  \varepsilon^4 \tau^2 (b-a) \sum_{l \in \mathcal{T}_M} (1 + |\mu_l|^2) |\widetilde{(f(\psi(t_n)))}_l - \widetilde{(f(\psi^n))}_l|^2 \nn \\
& \ +\varepsilon^4 \tau^4 (b-a) \sum_{l \in \mathcal{T}_M} (1 + |\mu_l|^2) |\widetilde{(\delta_t^- f(\psi(t_n)))}_l - \widetilde{(\delta_t^- f(\psi^n))}_l|^2 \nn \\
\lesssim & \  \varepsilon^4 \tau^2 \left\|I_M (f(\psi(t_n))) - I_M (f(\psi^n))\right\|_{H^1}^2 \nn\\
&\ + \varepsilon^4 \tau^4 \left\|I_M (\delta_t^- f(\psi(t_n))) - I_M (\delta_t^- f(\psi^n))\right\|_{H^1}^2.
\end{align}

We are going to estimate the two terms in the RHS of \eqref{eq:eta_estimate_1_ch4}. According to our assumption $\|\psi^n\|_{l^{\infty}} \le M_1+1$, $f$ is locally Lipschitz on interval $[-(M_1+1), M_1+1]$ with $C_{M_1}$ being the Lipschitz coefficients. Thus, we have
\begin{align}
&\left\|I_M (f(\psi(t_n))) - I_M (f(\psi^n))\right\|_{L^2}^2 = h\sum_{j=0}^{M-1} |f(\psi(x_j,t_n)) - f(\psi_j^n)|^2 \nn \\
&\le C_{M_1} h \sum_{j=0}^{M-1} |\psi(x_j,t_n) - \psi_j^n|^2  = C_{M_1} \left\|I_M (\psi(t_n)) - I_M\psi^n\right\|_{L^2}^2 \nn \\
&\le 2C_{M_1} \left(\left\|I_M \psi(t_n) - P_M \psi(t_n)\right\|_{L^2}^2 + \left\|P_M \psi(t_n) - I_M \psi^n\right\|_{L^2}^2\right) \nn \\
&\lesssim h^{2m}+   \left\|e^n_M(x)\right\|_{L^2}^2,
\end{align}
To estimate $\left\|\nabla (I_M (f(\psi(t_n))) - I_M (f(\psi^n)))\right\|_{L^2}$, by Lemma \ref{lmm:h1_ch4}, it remains to estimate 
\begin{equation*}
\left\|\delta_x^+ (f(\psi(x_j,t_n)) - f(\psi^n_j))\right\|_{l^2}.
\end{equation*}
It can be written as
\begin{align*}
\delta_x^+ (f(\psi(x_j,t_n)) - f(\psi^n_j)) =& \int_0^1 |\phi_{1,j}(\theta)|^2 \delta_x^+ \psi(x_j,t_n) -|\phi_{2,j}(\theta)|^2 \delta_x^+ \psi^n_j d\theta \\
        &+\int_0^1  (\phi_{1,j}(\theta))^2 \delta_x^+ \overline{\psi(x_j,t_n)} - (\phi_{2,j}(\theta))^2 \delta_x^+ \overline{\psi^n_j} d\theta,
\end{align*}
where for $0 \le \theta \le 1$,
\begin{align*}
&\phi_{1, j}(\theta) = \theta \psi(x_{j+1}, t_n) + (1-\theta)\psi(x_j, t_n),\quad \phi_{2, j}(\theta) = \theta \psi^n_{j+1} + (1-\theta)\psi^n_{j}.
\end{align*}
With the assumption $\|\psi^n\|_{l^{\infty}} \le M_1+1$, it is obvious that $|\phi_{1,j}(\theta)|, |\phi_{2,j}(\theta)|\le M_1 +1$. In addition, the assumption (A) implies $\|\delta_x^+ \psi(x_j,t_n)\|_{l^{\infty}} \lesssim \|\partial_x \psi(x,t_n)\|_{L^{\infty}} \lesssim 1.$
Then, by the locally Lipschitz property of $|\cdot|^2$, for $ j \in \mathcal{T}^0_M$, we have
\begin{align*}
&\left|\int_0^1 |\phi_{1,j}(\theta)|^2 \delta_x^+ \psi(x_j, t_n) -|\phi_{2,j}(\theta)|^2 \delta_x^+ \psi^n_j d\theta\right| \\
&\le \left|\int_0^1 \left(|\phi_{1,j}(\theta)|^2 - |\phi_{2,j}(\theta)|^2\right) \delta_x^+ \psi(x_j, t_n) d \theta\right| + \left|\int_0^1 |\phi_{2,j}(\theta)|^2 \left(\delta_x^+ \psi(x_j, t_n) - \delta_x^+ \psi^n_j\right) d\theta\right|\\        
&\lesssim \left|\delta_x^+ \psi(x_j, t_n)\right|\int_0^1\left|\phi_{1,j}(\theta) - \phi_{2,j}(\theta)\right| d \theta + \left|\phi_{2,j}(\theta)|^2 \int_0^1 |\delta_x^+ \psi(x_j, t_n) - \delta_x^+ \psi^n_j\right| d \theta \\
&\lesssim \left|\psi(x_{j+1},t_n) - \psi^n_{j+1}\right| + \left|\psi(x_{j},t_n) - \psi^n_{j}\right| + \left|\delta_x^+ \psi(x_j, t_n) - \delta_x^+ \psi^n_j\right|.
\end{align*}
Similarly, for the second term, $(\cdot)^2$ is locally Lipschitz, so we have for $ j \in \mathcal{T}^0_M$
\begin{align*}
&\left|\int_0^1  (\phi_{1,j}(\theta))^2 \delta_x^+ \overline{\psi(x_j, t_n)} - (\phi_{2,j}(\theta))^2 \delta_x^+ \overline{\psi^n_j} d\theta\right| \\
&\lesssim \left|\psi(x_{j+1},t_n) - \psi^n_{j+1}\right| + \left|\psi(x_{j},t_n) - \psi^n_{j}\right| + \left|\delta_x^+ \psi(x_j, t_n) - \delta_x^+ \psi^n_j\right|.
\end{align*}
Combining these two estimates, noticing Lemma \ref{lmm:h1_ch4}, we obtain
\begin{align*}
&\left\|\delta_x^+ (f(\psi(x_j, t_n)) - f(\psi^n_j)\right\|_{l^2} \\
&\lesssim \left\|\psi(x_{j},t_n) - \psi^n_{j}\right\|_{l^2} + \left\|\delta_x^+ \psi(x_j, t_n) - \delta_x^+ \psi^n_j\right\|_{l^2} \\
&\lesssim  \left\|I_M\psi(x, t_n) - I_M \psi^n\right\|_{L^2} + \left\|\nabla I_M \psi(x, t_n)  - \nabla I_M \psi^n \right\|_{L^2}\\
&\lesssim  \left\|I_M \psi(x,t_n) - P_M \psi(x,t_n)\right\|_{L^2} + \left\|e_M^{n}\right\|_{L^2} + \left\|\nabla I_M \psi(x,t_n)  - \nabla P_M \psi(x,t_n)\right\|_{L^2} + \left\|\nabla e_M^{n}\right\|_{L^2} \\
&\lesssim \left \|e_M^{n}\right\|_{H^1} + h^{m-1}.
\end{align*}
Therefore, by Lemma \ref{lmm:h1_ch4}, for the $H^1$ semi-norm we have
\begin{align*}
\left\|\nabla (I_M (f(\psi(t_n))) - I_M (f(\psi^n)))\right\|_{L^2} &\lesssim \left\|\delta_x^+ (f(\psi(x_j, t_n)) - f(\psi^n_j))\right\|_{l^2} \\
&\lesssim \left\|e_M^{n}\right\|_{H^1} + h^{m-1}.
\end{align*}
Combining the estimates in $L^2$-norm, we obtain
\begin{equation*}
\left\|I_M (f(\psi(t_n))) - I_M (f(\psi^n))\right\|_{H^1} \lesssim \|e_M^{n}\|_{H^1} + h^{m-1}.
\end{equation*}
As a consequence, we have for $n \ge 1$,
\begin{align*}
&\left\|I_M (\delta_t^- f(\psi(t_n))) - I_M (\delta_t^- f(\psi^n))\right\|_{H^1} \\
&\le \tau^{-1} \left(\left\|I_M (f(\psi(t_n))) - I_M (f(\psi^n))\right\|_{H^1} + \left\|I_M (f(\psi(t_{n-1}))) - I_M (f(\psi^{n-1}))\right\|_{H^1}\right)\\
&\lesssim  \tau^{-1} \left(\left\|e_M^{n}\right\|_{H^1} + \left\|e_M^{n-1}\right\|_{H^1} + h^{m-1}\right).
\end{align*}
Substituting the above two estimates into \eqref{eq:eta_estimate_1_ch4}, we obtain
\begin{equation*}
\left\|\eta^{n,\pm}(x)\right\|_{H_1}^2 \lesssim \eps^4 \tau^2 \left(\left\|e_M^{n}\right\|_{H^1}^2 + \left\|e_M^{n-1}\right\|_{H^1}^2 + h^{2m-2}\right),
\end{equation*}
which completes the proof of the error bound \eqref{eq:eta_lmm_ch4}.
\end{proof}

\subsection{Proof of Theorem \ref{thm:main_ch4}}
First, we consider the cases $n=0$ and $n=1$. When $n=0$, $\left\|e^0_M(x)\right\|_{H^1} = \left\|P_M \psi_0 - I_M \psi_0\right\|_{H^1} \lesssim h^{m-1}$. When $n=1$, we have the decomposition
\begin{align*}
\widetilde{(e_M^{1})}_l =&\ \widetilde{(\xi^0)}_l + c^0_l\left( \widehat{(\psi_0)}_l -\widetilde{(\psi_0)}_l\right) + d^0_l \left(\widehat{(\psi_1)}_l-\widetilde{(\psi_1)}_l\right)\\
        =&\ e^{i \beta^+_l \tau} \widetilde{(e_M^{0,+})}_l + e^{i \beta^-_l \tau} \widetilde{(e_M^{0,-})}_l + e^{i \beta^+_l \tau} \widetilde{(\xi^{0,+})}_l \\
        & \ - e^{i \beta^-_l \tau} \widetilde{(\xi^{0,-})}_l + e^{i \beta^+_l \tau} \widetilde{(\eta^{0,+})}_l - e^{i \beta^-_l \tau} \widetilde{(\eta^{0,-})}_l,
\end{align*}
where $\widetilde{(\eta^{0,\pm})}_l = 0$ and $e_M^{0,\pm}(x) = \sum_{l\in \mathcal{T}_M} \widetilde{(e_M^{0,\pm})}_l e^{i \mu_l (x-a)} \in X_M$ is defined as 
\begin{align*}
\widetilde{(e_M^{0,+})}_l &= - \beta_l^{-1}\left(\beta_l^- \left( \widehat{(\psi_0)}_l -\widetilde{(\psi_0)}_l\right) + i \left(\widehat{(\psi_1)}_l-\widetilde{(\psi_1)}_l\right)\right),\\
\widetilde{(e_M^{0,-})}_l &=  \beta_l^{-1} \left(\beta_l^+ \left( \widehat{(\psi_0)}_l -\widetilde{(\psi_0)}_l\right) + i \left(\widehat{(\psi_1)}_l-\widetilde{(\psi_1)}_l\right)\right),
\end{align*}
so $\widetilde{(e_M^{0})}_l = \widetilde{(e_M^{0,+})}_l + \widetilde{(e_M^{0,-})}_l$. Since  $\beta_l^{-1} \beta_l^{\pm}, \beta_l^{-1} \lesssim 1$ for all $l \in \mathcal{T}_M$, by Bessel equality, we have
\begin{equation}
\left\|e_M^{0,\pm}(x)\right\|_{H^1} \lesssim \left\|P_M \psi_0 - I_M \psi_0\right\|_{H^1} + \left\|P_M \psi_1 - I_M \psi_1\right\|_{H^1}\lesssim h^{m-1} ,
\end{equation}
which implies
\begin{align*}
\left\|e_M^{1}(x)\right\|_{H^1} &\lesssim \left\|e_M^{0,+}(x)\right\|_{H^1} + \left\|e_M^{0,-}(x)\right\|_{H^1}+\left \|\xi^{0,+}(x)\right\|_{H^1}  + \left\|\xi^{0,-}(x)\right\|_{H^1} \\
        &\lesssim h^{m-1} + \eps^2 \tau^3.
\end{align*}
In 1D, by discrete Sobolev inequality, we have
\begin{equation}
\left\|e^n(x)\right\|_{L^{\infty}}^2 \lesssim \left\|e^n(x)\right\|_{L^2} \left\|\nabla e^n(x)\right\|_{L^2} \lesssim \left\|e^n(x)\right\|_{H^1}^2,
\end{equation}
and
\begin{equation}
\left\|\psi^n\right\|_{l^{\infty}} \leq \left\|\psi(x, t_n)\right\|_{L^{\infty}}+\left\|e^n(x)\right\|_{L^{\infty}}\leq M_1+\left\|e^n(x)\right\|_{L^{\infty}}.
\end{equation}
For $n=0,1$, we have proven $\left\|e^n(x)\right\|_{H^1} \lesssim h^{m-1} + \left\|e_M^n(x)\right\|_{H^1}\lesssim h^{m-1}+\eps^2\tau^3$, so there exist two constants $h_0,\tau_0>0$, when $0<h<h_0$ and  $0<\tau<\tau_0$, we have
\begin{equation}
\left\|\psi^n\right\|_{l^{\infty}}  \le M_1 + 1,
\end{equation}
which means the error bound \eqref{eq:error_main_ch4} holds for $n=0, 1$.

Next, we are going to adopt the mathematical induction to proceed the proof. Assuming the error bound \eqref{eq:error_main_ch4} holds for all $0 \le n \le q \le \frac{T/\eps^{\beta}}{\tau}-1$, and for all $0 \le n \le q$, the following error decomposition holds
\begin{align}
\widetilde{(e_M^n)}_l =&\ e^{i \beta^+_l n \tau} \widetilde{(e_M^{0,+})}_l + e^{i \beta^-_l n \tau} \widetilde{(e_M^{0,-})}_l  \nn \\
&\ + \sum_{k=0}^{n-1} \biggl( e^{i \beta^+_l (n-k) \tau} \widetilde{(\xi^{k,+})}_l - e^{i \beta^-_l(n-k) \tau} \widetilde{(\xi^{k,-})}_l \nn \\
&\ + e^{i \beta^+_l (n-k) \tau} \widetilde{(\eta^{k,+})}_l - e^{i \beta^-_l(n-k) \tau} \widetilde{(\eta^{k,-})}_l\biggr), 
\label{eq:error_decomposition_ch4}
\end{align}
we will prove \eqref{eq:error_main_ch4} and \eqref{eq:error_decomposition_ch4} holds for $n = q + 1$.

For any $a,b\in \mathbb{C}$ and $k\ge 0$, we have the equality
\begin{align}
&\left(e^{i\beta_l^+\tau} + e^{i\beta_l^-\tau}\right)\left(e^{i \beta^+_l k \tau} a + e^{i \beta^-_l k \tau} b\right) - e^{i(\beta_l^+ + \beta_l^-)\tau}\left(e^{i \beta^+_l (k-1) \tau} a + e^{i \beta^-_l(k-1) \tau} b\right) \nn  \\
        &= e^{i \beta^+_l (k+1) \tau} a  +e^{i \beta^-_l (k+1) \tau} b.
\label{eq:ab}
\end{align}
Then in \eqref{eq:em_decomposition_ch4}, by the definition of $c_l$ and $d_l$ in \eqref{eq:coefficients_ch4}, combining the decomposition of $\xi$ and $\eta$ and applying the equality \eqref{eq:ab}, we obtain
\begingroup
\allowdisplaybreaks
\begin{align*}
\widetilde{(e_M^{q+1})}_l =&\  \left(e^{i\beta_l^+\tau} + e^{i\beta_l^-\tau}\right) \widetilde{(e_M^q)}_l -e^{i(\beta_l^+ + \beta_l^-)\tau} \widetilde{(e_M^{q-1})}_l  \\
&\ +e^{i\beta_l^+\tau}\widetilde{(\xi^{q,+})}_l - e^{i\beta_l^-\tau}\widetilde{(\xi^{q,-})}_l - e^{i(\beta_l^+ + \beta_l^-)\tau}\left(\widetilde{(\xi^{q-1,+})}_l - \widetilde{(\xi^{q-1,-})}_l\right)\\
&\ + e^{i\beta_l^+\tau}\widetilde{(\eta^{q,+})}_l - e^{i\beta_l^-\tau}\widetilde{(\eta^{q,-})}_l -e^{i(\beta_l^+ + \beta_l^-)\tau}\left(\widetilde{(\eta^{q-1,+})}_l - \widetilde{(\eta^{q-1,-})}_l\right)\\
= &\ e^{i \beta^+_l (q+1) \tau} \widetilde{(e_M^{0,+})}_l + e^{i \beta^-_l (q+1) \tau} \widetilde{(e_M^{0,-})}_l \\
&\ +\sum_{k=0}^{q-2} \biggl( e^{i \beta^+_l (q+1-k) \tau} \widetilde{(\xi^{k,+})}_l - e^{i \beta^-_l(q+1-k) \tau} \widetilde{(\xi^{k,-})}_l \\
&\  + e^{i \beta^+_l (q+1-k) \tau} \widetilde{(\eta^{k,+})}_l - e^{i \beta^-_l(q+1-k) \tau} \widetilde{(\eta^{k,-})}_l\biggr)\\
&\ + \left(e^{i\beta_l^+\tau} + e^{i\beta_l^-\tau}\right)\left(e^{i \beta^+_l \tau} \widetilde{(\xi^{q-1,+})}_l - e^{i \beta^-_l \tau} \widetilde{(\xi^{q-1,-})}_l\right) \\&-e^{i(\beta_l^+ + \beta_l^-)\tau} (\widetilde{(\xi^{q-1,+})}_l - \widetilde{(\xi^{q-1,-})}_l)\\
&\ +(e^{i\beta_l^+\tau} + e^{i\beta_l^-\tau} )\left(e^{i \beta^+_l \tau} \widetilde{(\eta^{q-1,+})}_l - e^{i \beta^-_l \tau} \widetilde{(\eta^{q-1,-})}_l\right) \\
&\ -e^{i(\beta_l^+ + \beta_l^-)\tau}\left(\widetilde{(\eta^{q-1,+})}_l - \widetilde{(\eta^{q-1,-})}_l\right)\\
&\ +e^{i\beta_l^+\tau}\widetilde{(\xi^{q,+})}_l - e^{i\beta_l^-\tau}\widetilde{(\xi^{q,-})}_l + e^{i\beta_l^+\tau}\widetilde{(\eta^{q,+})}_l - e^{i\beta_l^-\tau}\widetilde{(\eta^{q,-})}_l, \\
\end{align*}
\endgroup
which implies
\begin{align*}
\widetilde{(e_M^{q+1})}_l= &\ e^{i \beta^+_l (q+1) \tau} \widetilde{(e_M^{0,+})}_l + e^{i \beta^-_l (q+1) \tau} \widetilde{(e_M^{0,-})}_l \\
&\ +\sum_{k=0}^{q} \biggl( e^{i \beta^+_l (q+1-k) \tau} \widetilde{(\xi^{k,+})}_l - e^{i \beta^-_l(q+1-k) \tau} \widetilde{(\xi^{k,-})}_l \\
&\ + e^{i \beta^+_l (q+1-k) \tau} \widetilde{(\eta^{k,+})}_l - e^{i \beta^-_l(q+1-k) \tau} \widetilde{(\eta^{k,-})}_l\biggr),\\
\end{align*}
i.e., the decomposition \eqref{eq:error_decomposition_ch4} still holds for $n = q + 1$. Then by Cauchy inequality, we have
\begin{align*}
\left|\widetilde{(e_M^{q+1})}_l\right|^2 & \le  6\biggl( \left|\widetilde{(e_M^{0,+})}_l\right|^2 +  \left|\widetilde{(e_M^{0,-})}_l\right|^2 \\
        &\quad  \ + (q+1) \sum_{k=0}^q \left(\left|\widetilde{(\xi^{k,+})}_l\right|^2 + \left|\widetilde{(\xi^{k,-})}_l\right|^2 + \left|\widetilde{(\eta^{k,+})}_l\right|^2 + \left|\widetilde{(\eta^{k,-})}_l\right|^2\right) \biggr).
\end{align*}
Combining Lemma \ref{lmm:local_truncation_ch4} and Lemma \ref{lmm:stability_ch4}, by Bessel equality, we have
\begin{align*}
\left\|e_M^{q+1}(x)\right\|_{H^1}^2 \le & \ 6\biggl(\left\|e_M^{0,+}(x)\right\|_{H^1}^2 +  \left\|e_M^{0,-}(x)\right\|_{H^1}^2+\left(q+1\right) \sum_{k=0}^q \Big(\left\|\xi^{k,+}(x)\right\|_{H^1}^2   \\
        &\  + \left\|\xi^{k,-}(x)\right\|_{H^1}^2+ \left\|\eta^{k,+}(x)\right\|_{H^1}^2 + \left\|\eta^{k,-}(x)\right\|_{H^1}^2 \Big)\biggr)\\
        \le & \  C_1 h^{2m-2} + C_2 \left(q+1\right)^2 \varepsilon^4 \tau^2 \left(\tau^4 + h^{2m-2}\right) \\
&\  + C_3 \left(q+1\right) \varepsilon^4 \tau^2 \sum_{k=0}^q \left(\left\|e^k_M\right\|_{H^1}^2  + h^{2m-2}\right)\\
        \le &\ C_0 \left(\varepsilon^{4-2\beta} \tau^4 + h^{2m-2}\right) + C_3 \varepsilon^{4-\beta} \tau \sum_{k=0}^q \left\|e^k_M\right\|_{H^1}^2.
\end{align*}
By the condition $0\le \beta \le 2$ and $0<\varepsilon \le 1$, we apply discrete Gronwall inequality to get
\begin{equation}
\left\|e_M^{q+1}(x)\right\|_{H^1}^2 \le C_T \left(\varepsilon^{4-2\beta} \tau^4 + h^{2m-2}\right), \quad 1 \le q \le \frac{T/\varepsilon^{\beta}}{\tau},
\end{equation}
where $C_T$ is a constant independent of $\tau,h,\varepsilon$ and $\beta$. Therefore, the first inequality in \eqref{eq:error_main_ch4} still holds for $n = q + 1$. By triangle inequality and discrete Sobolev inequality in 1D, there exist two constants $h_0>0$ and $\tau_0>0$ sufficiently small such that when $0<h<h_0$ and $0<\tau<\tau_0$, we have
\begin{equation}
\left\|\psi^{q+1}\right\|_{l^{\infty}} \le \left\|\psi(x, t_q)\right\|_{L^{\infty}} + \left\|e^{q+1}(x)\right\|_{L^{\infty}} \le 1+M_1.
\end{equation}
Therefore, by the method of mathematical induction, the proof of Theorem \ref{thm:main_ch4} is completed.

\section{Numerical results}
 In this section, we present the numerical results for the EWI-FP scheme \eqref{eq:ewi_fp} with \eqref{eq:ewi_fp_scheme_ch4_1}--\eqref{eq:ewi_fp_scheme_ch4_2} for the NLSW with weak nonlinearity \eqref{eq:NLSW_wl_1D_ch4}. In the following numerical experiments, we choose $\alpha =1$, the computational domain $\Omega = (-\pi,\pi)$, and the initial data as
\begin{equation}
\psi_0(x) = \frac{1}{2 + \cos^2(x) + \sin(x)},\quad \psi_1(x) = \frac{1}{2 + \sin^2(x) + \cos(x)}.
\end{equation}
The numerical simulations are presented on the time interval $[0, T/\eps^{\beta}]$ with $0 \le \beta \le 2$ and $T=1$ fixed. Here, we study the following three cases with different $\beta$: 

Case I. Fixed time dynamics up to the time at $O(1)$, i.e., $\beta =0$.

Case II. Intermediate long-time dynamics up to the time at  $O(\varepsilon^{-1})$, i.e., $\beta =1$.

Case III. Long-time dynamics up to the time at  $O(\varepsilon^{-2})$, i.e., $\beta =2$.

Since the exact solution of the NLSW is unknown, we use the proposed EWI-FP scheme with a very fine mesh $h_e = \pi/64$ and a very small time step $\tau_e = 5 \times 10^{-4}$ to get the `reference' solution numerically. In order to quantify the numerical errors, we measure the $H^1$-norm of $e(\cdot, t_n = 1/\eps^{\beta}) = \psi(x, t_n) - I_M\psi^n$. 

Spatial and temporal errors are displayed at $t = 1/\eps^{\beta}$ with different $\eps$ and $\beta$. For the test of spatial errors, we fix the time step size as $\tau =  5 \times 10^{-4}$ such that the temporal errors can be ignored. Table \ref{tabel:spatial_ch4} shows the long-time spatial errors for $\beta = 0$, $\beta = 1$ and $\beta = 2$, which indicates that the EWI-FP scheme is uniformly spectral accurate in space for any $0 <\varepsilon\le 1$ and $0 \le \beta \le 2$. 

\begin{table}[htb]
\caption{Spatial errors of the EWI-FP scheme for the NLSW \eqref{eq:NLSW_wl_1D_ch4} with different $\beta$ and $\varepsilon$.}
\renewcommand\arraystretch{1.1}
    \centering
    \resizebox{\textwidth}{!}{\begin{tabular}{c|ccccc}
        \hline
        & $\left\|e(\cdot, t = 1/\varepsilon^{\beta})\right\|_{H^1}$ & $h=\pi/4$ &$h=\pi/8$ &$h=\pi/16$ &$h=\pi/32$ \\
        \hline
        \multirow{5}{*}{$\beta=0$}&$\varepsilon = 1$ & 1.81E-1&	5.69E-3&	8.84E-5&	7.01E-10        \\
         &$\varepsilon = 1/2$& 1.28E-1&	6.57E-3&	5.95E-5&	6.43E-10      \\
         &$\varepsilon = 1/2^2$& 1.08E-1&	7.53E-3&	5.27E-5&	6.34E-10         \\
         &$\varepsilon = 1/2^3$& 1.04E-1&	7.74E-3&	5.10E-5&	6.29E-10         \\
         &$\varepsilon = 1/2^4$& 1.02E-1&	7.79E-3&	5.05E-5&	6.27E-10         \\
        \hline
        \multirow{5}{*}{$\beta=1$}&$\varepsilon = 1$ & 1.81E-1&	5.69E-3&	8.84E-5&	7.01E-10        \\
        &$\varepsilon = 1/2$& 1.28E-1&	7.39E-3&	3.28E-5&	1.22E-10        \\
        &$\varepsilon = 1/2^2$& 8.86E-2&	1.05E-2&	3.89E-5&	2.29E-10        \\
        &$\varepsilon = 1/2^3$& 3.94E-2&	1.20E-2&	6.05E-5&	5.17E-10        \\
        &$\varepsilon = 1/2^4$& 7.66E-2&	6.84E-3&	6.42E-6&	4.07E-10        \\
        \hline
        \multirow{5}{*}{$\beta=2$}&$\varepsilon = 1$ & 1.81E-1&	5.69E-3&	8.84E-5&	7.01E-10        \\
        &$\varepsilon = 1/2$& 8.70E-2&	1.22E-2&	5.04E-5&	2.50E-10        \\
        &$\varepsilon = 1/2^2$& 8.60E-2&	8.43E-3&	9.12E-6&	4.10E-10        \\
        &$\varepsilon = 1/2^3$& 1.04E-1&	4.64E-3&	3.43E-5&	5.74E-10        \\
        &$\varepsilon = 1/2^4$& 1.15E-1&	1.14E-2&	5.78E-5&	2.81E-10        \\
        \hline
    \end{tabular}}
\label{tabel:spatial_ch4}
\end{table}

For the temporal errors, a very fine mesh size $h = \pi/64$ is chosen such that the spatial errors can be neglected. Figures \ref{fig:beta0}--\ref{fig:beta2} depict the temporal errors of the EWI-FP scheme with different $\eps$ and $\tau$ for $\beta =0$, $\beta = 1$ and $\beta = 2$, respectively. From these figures and additional numerical results not shown here for brevity, we have the following observations: (i) In time, for any fixed $\eps = \eps_0 > 0$, the EWI-FP scheme is second-order accurate (cf. each line in Figures \ref{fig:beta0}(a)--\ref{fig:beta2}(a)). (ii) When $\beta = 0$, the temporal error behaves like $O(\eps^2\tau^2)$ (cf. Figure \ref{fig:beta0}(b)). Figure \ref{fig:beta1}(b) and Figure \ref{fig:beta2}(b) show that the temporal error is at $O(\eps\tau^2)$ for $\beta = 1$ and $O(\tau^2)$ for $\beta = 2$, respectively. (iii)  Our numerical results confirm the uniform error bounds given in the Theorem \ref{thm:main_ch4} and suggest that they are sharp.

\begin{figure}[ht!]
\begin{minipage}{0.49\textwidth}
\centerline{\includegraphics[width=7.5cm,height=6cm]{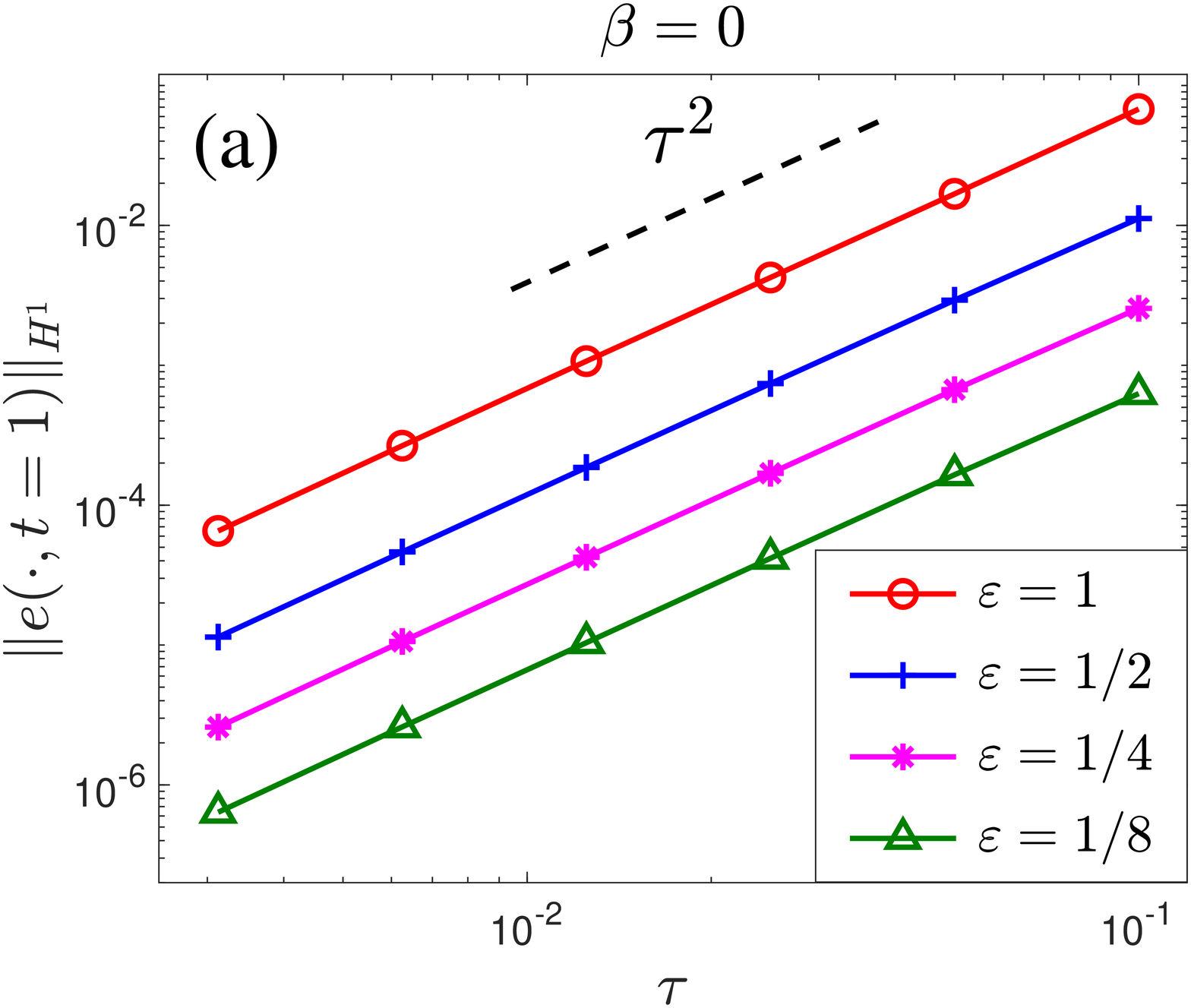}}
\end{minipage}
\begin{minipage}{0.49\textwidth}
\centerline{\includegraphics[width=7.5cm,height=6cm]{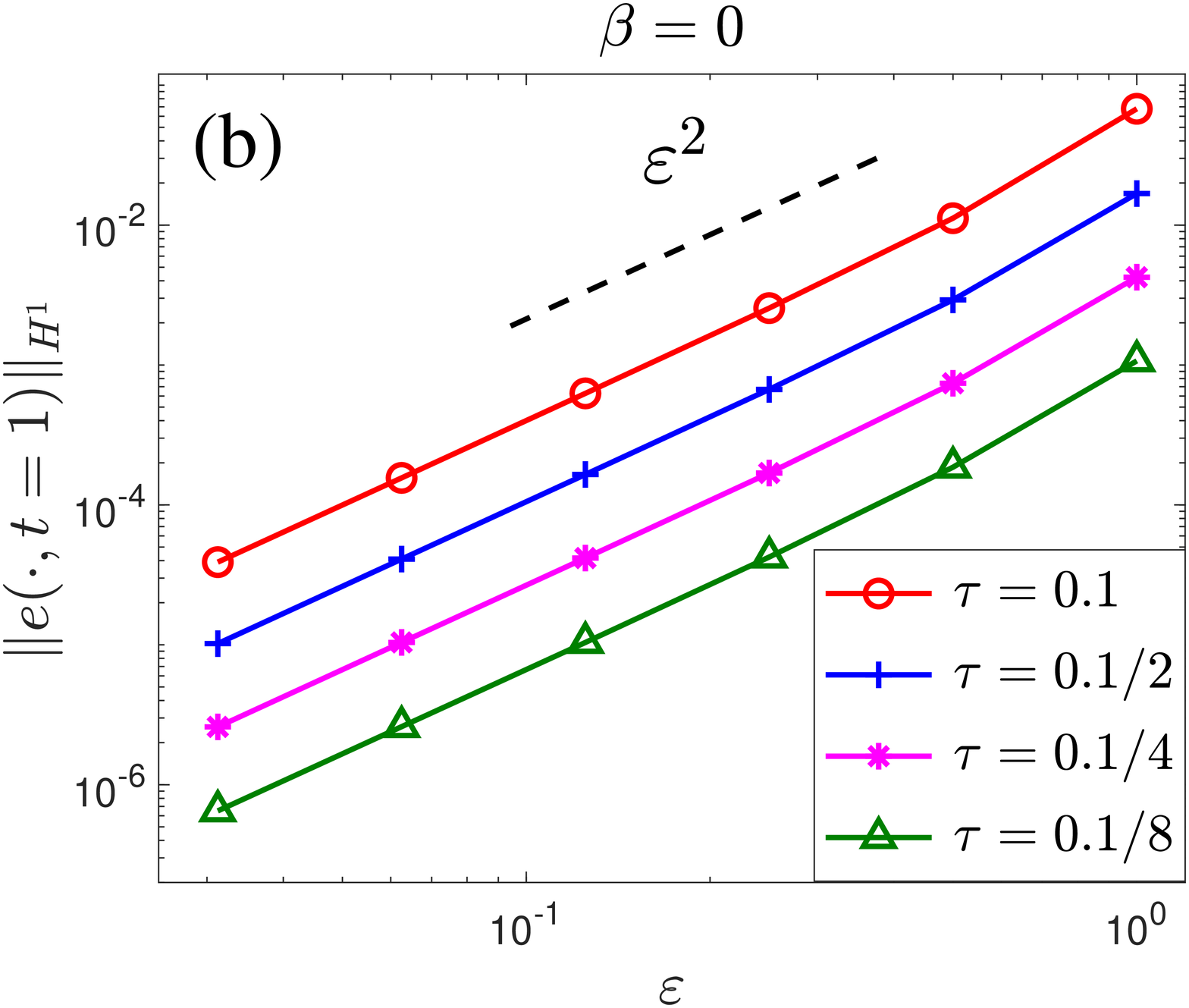}}
\end{minipage}
\caption{Temporal errors of the EWI-FP scheme for the NLSW \eqref{eq:NLSW_wl_1D_ch4} with $\beta = 0$  for (a) different $\eps$; and  (b) different $\tau$.}
\label{fig:beta0}
\end{figure}

\begin{figure}[ht!]
\begin{minipage}{0.49\textwidth}
\centerline{\includegraphics[width=7.5cm,height=6cm]{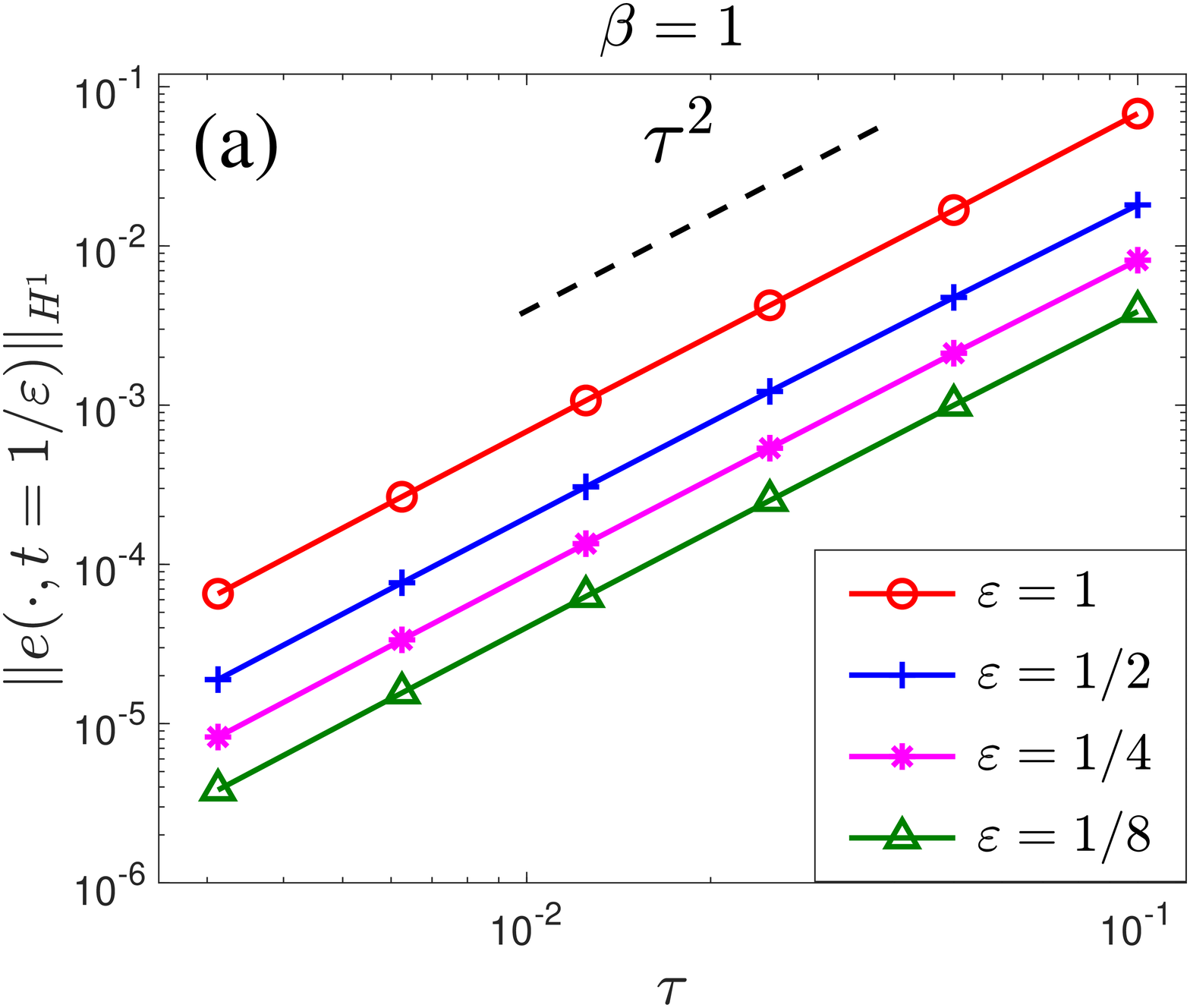}}
\end{minipage}
\begin{minipage}{0.49\textwidth}
\centerline{\includegraphics[width=7.5cm,height=6cm]{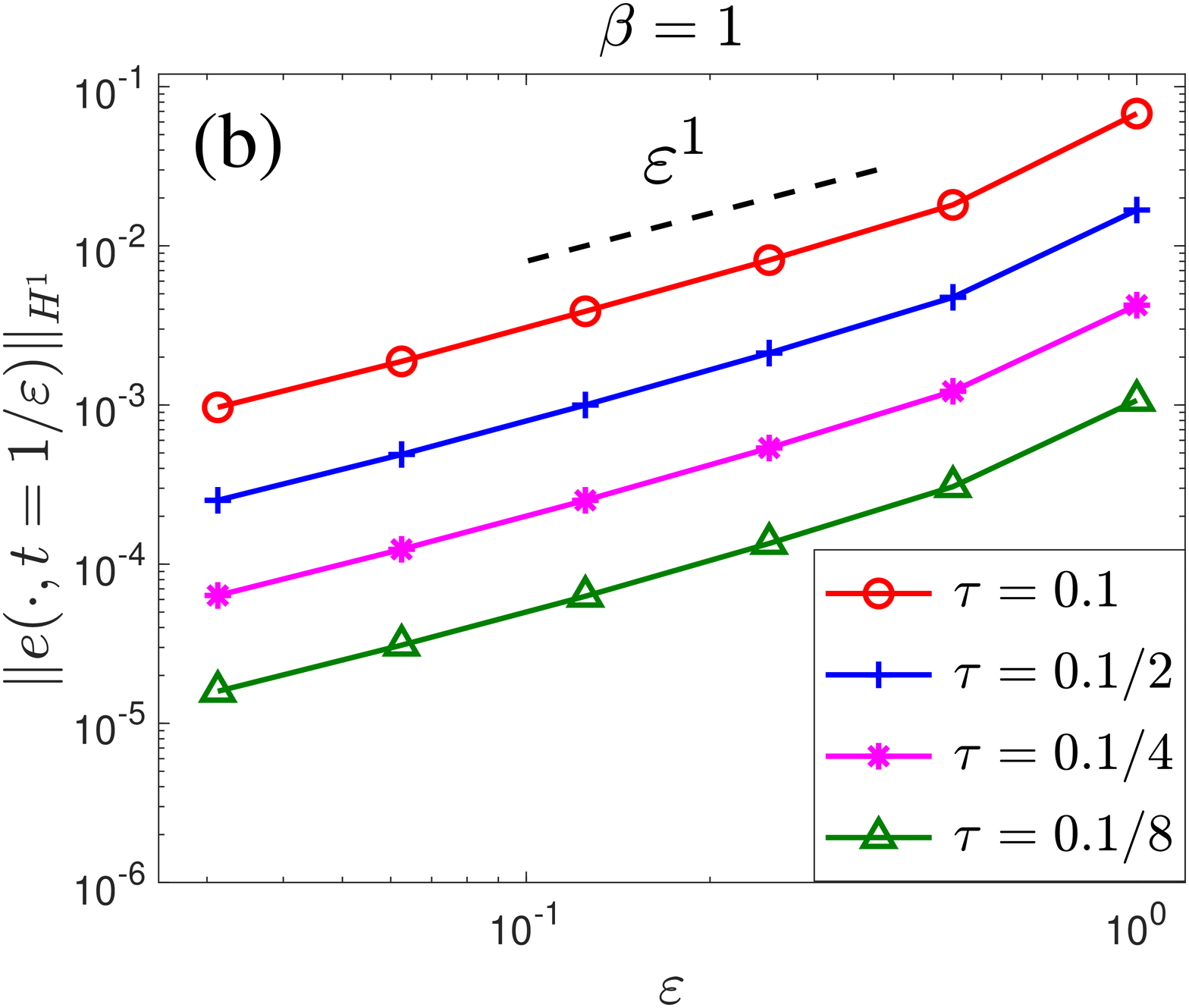}}
\end{minipage}
\caption{Temporal errors of the EWI-FP scheme for the NLSW \eqref{eq:NLSW_wl_1D_ch4} with $\beta = 1$  for (a) different $\eps$; and  (b) different $\tau$.}
\label{fig:beta1}
\end{figure}

\begin{figure}[ht!]
\begin{minipage}{0.49\textwidth}
\centerline{\includegraphics[width=7.5cm,height=6cm]{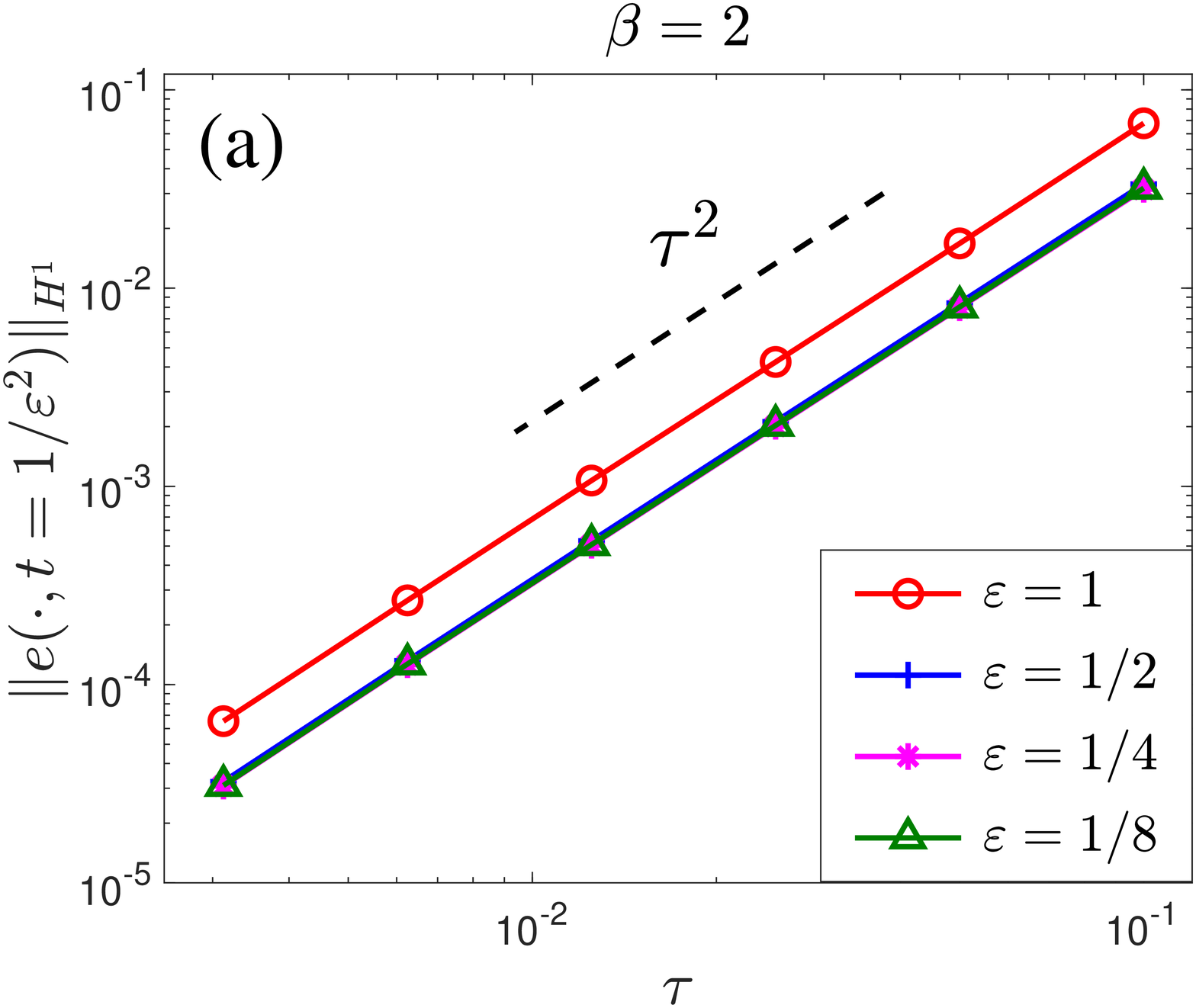}}
\end{minipage}
\begin{minipage}{0.49\textwidth}
\centerline{\includegraphics[width=7.5cm,height=6cm]{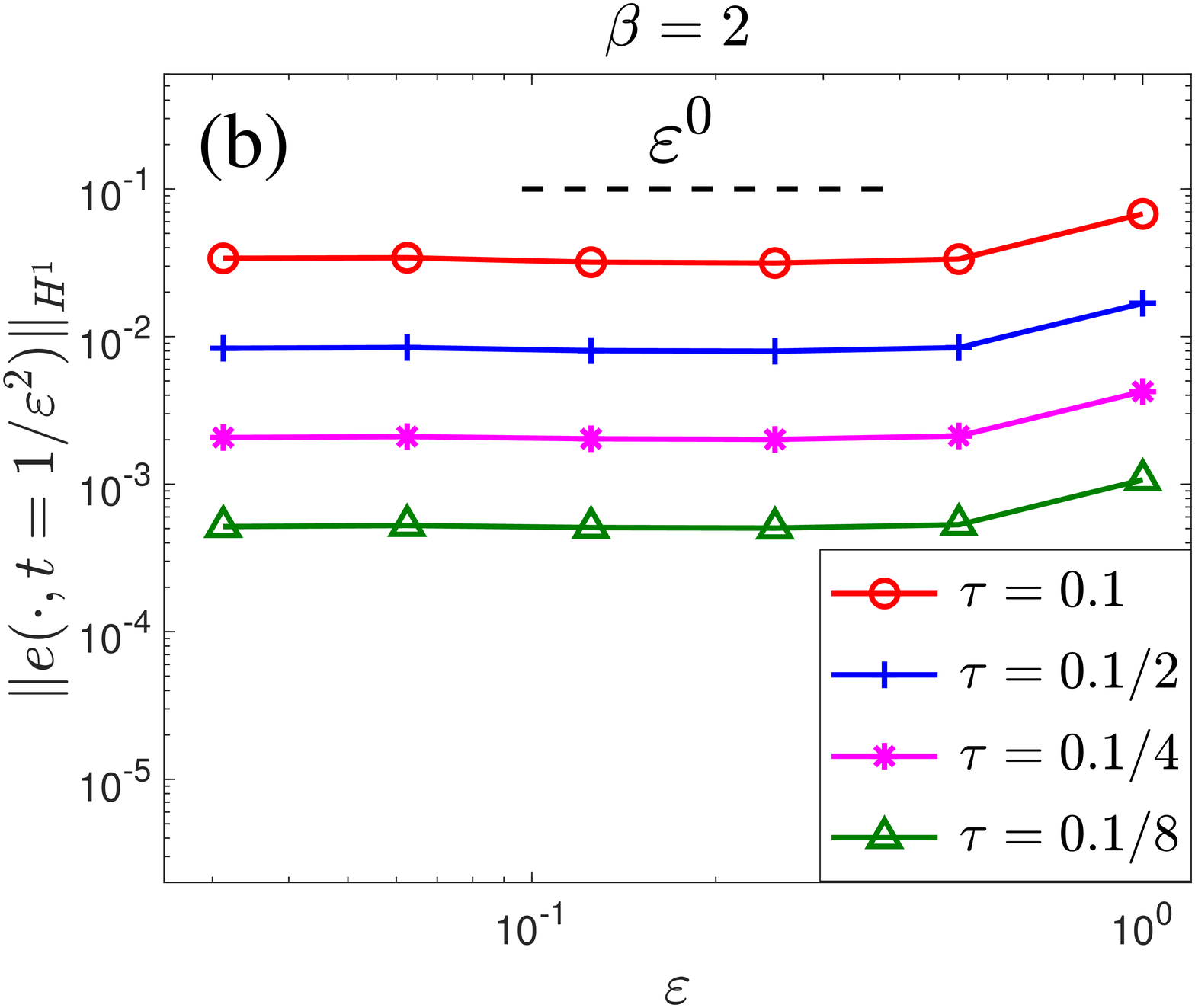}}
\end{minipage}
\caption{Temporal errors of the EWI-FP scheme for the NLSW \eqref{eq:NLSW_wl_1D_ch4} with $\beta = 2$  for (a) different $\eps$; and  (b) different $\tau$.}
\label{fig:beta2}
\end{figure}

\section{Conclusions}
The exponential wave integrator Fourier pseudospectral (EWI-FP) method was applied to discretize the nonlinear Schr\"odinger equation with wave operator (NLSW) with weak nonlinearity, where the strength of the nonlinearity is characterized by $\eps^{2p}$ with $0 < \eps \leq 1$ a dimensionless parameter and $ p \in \mathbb{N}^+$. Uniform error bound of the EWI-FP method was rigorously carried out at $O(h^{m-1}+ \eps^{2p-\beta}\tau^2)$
 for the long-time dynamics of the NLSW up to the time  $t = T/\eps^{\beta}$ with $T > 0$ fixed and $0 \leq \beta \leq 2p$. Finally, numerical results were presented to confirm the error bounds and demonstrate that they are optimal and sharp.

\section*{Acknowledgments}
The authors would like to thank Professor Weizhu Bao for his valuable suggestions and comments. YF gratefully acknowledges support from the Ministry of Education of Singapore grant MOE-000357-01 and the European Research Council (ERC) under the European Union's Horizon 2020 research and innovation programme (grant agreement No. 850941). YY was partially supported by the Natural Science Foundation of China (grant agreement No. 11971007). Part of the work was done when YF was visiting the Department of Mathematics and the Institute for Mathematical Sciences at the National University of Singapore in 2023.


\end{document}